\documentclass{article}
\usepackage[a4paper]{geometry}
\usepackage[english]{babel}
\usepackage[utf8]{inputenc}
\usepackage{amsmath}
\usepackage[toc,page]{appendix}
\usepackage{amsthm}
\usepackage[small]{titlesec}
\usepackage{mathtools}
\usepackage{graphicx}
\usepackage{amsfonts}
\usepackage{amssymb}
\usepackage{mathrsfs}
\usepackage{xcolor}
\usepackage{esint}
\usepackage{relsize}
\usepackage{comment}
\usepackage{bigints}
\usepackage{dsfont}
\usepackage{geometry}
\usepackage{float}
\usepackage{bbm}

\numberwithin{equation}{section}
\theoremstyle{plain}
\newtheorem{thm}{Theorem}[section]

\newtheorem{lem}[thm]{Lemma}
\newtheorem{prop}[thm]{Proposition}
\theoremstyle{definition}
\newtheorem{defn}{Definition}
\newtheorem{remark}{Remark}

\theoremstyle{remark}

\newcommand{\R}{\mathbb{R}^3}

\newcommand{\Ll}{\mathscr{L}}

\newcommand{\M}{\mathscr{M}_{+}\left(\R \right)}

\newcommand{\vp}{\varphi}
\newcommand{\ve}{\varepsilon}

\newcommand{\D}{\mathscr{D}}

\begin{document}
\title{Homoenergetic solutions for the Rayleigh-Boltzmann equation 
: existence of a stationary non-equilibrium solution}

\author{Nicola Miele, Alessia Nota, Juan J. L. Vel\'azquez}
\date{\today }

\maketitle

\begin{abstract} 
 In this paper we consider a particular class of 
solutions of the linear Boltzmann-Rayleigh  equation, known in the nonlinear setting as Homoenergetic solutions. These solutions describe the dynamics of Boltzmann gases under the effect of different mechanical deformations.   Therefore, the long-time behaviour 
of these solutions cannot be described by Maxwellian distributions and it strongly depends on the homogeneity of the collision kernel of the equation. 
  
Here we focus on the paradigmatic case of simple shear 
deformations and in the case of cut-off collision kernels with homogeneity $\gamma\geq 0$, in particular covering the case of Maxwell molecules (i.e. $\gamma=0$) and hard potentials with $0\leq \gamma <1$.  We first prove a well-posedness result for this class of solutions in the space of non-negative Radon measures and then we rigorously prove 
the existence of a stationary solution under the non-equilibrium condition which is induced by the presence of the shear deformation. In the case of Maxwell molecules we prove that there is a different behaviour of the solutions for small and large values of the shear parameter.  
\end{abstract}

\tableofcontents

\section{Introduction}

In this paper we consider the Boltzmann-Rayleigh Equation, or simply linear Boltzmann  Equation, for the one-particle probability distribution $g : \mathbb{R}^+\times \R\times \R \rightarrow  \mathbb{R}^+$,  which reads 
\begin{align}
\partial_{t}g+w\cdot \partial_{x}g  &  =(\Ll g)\left(  w\right), 
	\;  \quad g=g\left(  t,x,w\right)  
\label{A0_0}
    \\
	(\Ll g)\left(  w\right)   &  =\int_{\mathbb{R}^{3}}dw_{\ast}\int_{S^{2}
	}   B\left( n \cdot  \omega, |w-w_*| \right)
	\left(  M_*'g_{\ast}^{\prime}-M_*g\right) d\omega,   \label{A0_1}
\end{align}
where $\partial_x$ denotes the usual gradient in Cartesian coordinates and $S^{2}$ is the unit sphere in $\mathbb{R}^{3}$. The position space is $\R$, the velocity space is $\R$, and the reference measure is the Maxwellian distribution with temperature $T=\beta^{-1}>0$, whose density with respect to the Lebesgue measure is denoted by $M_\beta(w_*)=\left(\frac \beta {2\pi}\right)^{\frac 3 2}e^{-\frac \beta {2} \vert w_* \vert^2}$. We set for simplicity $\beta=1$ and thus  $M(w_*)=\left(\frac{1}{2 \pi}\right)^{\frac{3}{2}}e^{-\frac{|w_*|^2}{2}}$ is the  normalized Maxwellian density function. Here  $n=n\left(w,w_{\ast}\right)  =\frac{\left(  w-w_{\ast}\right)}{\left\vert w- w_{\ast}\right\vert }$. The pair $(w,w_{\ast})$ denotes the velocity of the tagged moving particle and the velocity of one background particle before the collision and $(w^{\prime},w_{\ast
}^{\prime})$ denotes the corresponding pair of post-collisional velocities obtained by the
collision rule
\begin{equation}\label{eq:collisonRule}
    \begin{cases}
        w'=w-((w-w_*) \cdot \omega))\omega, & \\
        w_*'=w_*+((w-w_*) \cdot \omega))\omega.
    \end{cases}
\end{equation}
The unit vector  $\omega=\omega(w,V)$ bisects the angle between the
incoming relative velocity $V=w-w_{\ast}$ and the outgoing relative velocity
$V^{\prime}=w'-w_{\ast}^{\prime}$.
The collision kernel $B\left(n \cdot \omega,|w|\right)  $ is proportional to the cross section for the scattering problem
associated to the collision between two particles and it is assumed to be a non-negative and smooth function. Here we adopted the conventional notation in kinetic theory 
$g=g\left(t, x, w\right)  , \ g_{\ast
}=g\left( t, x, w_{\ast}\right)  , \ g^{\prime}=g\left( t, x, w^{\prime}\right)
, \ g_{\ast}^{\prime}=g\left( t, x,  w_{\ast}^{\prime}\right)  $. 

We notice that  \eqref{A0_0}-\eqref{A0_1} describes at mesoscopic level the dynamics of a tagged point particle in a gas at thermal equilibrium distributed in the whole space $\R$. This model is referred to
as the ideal Rayleigh gas, see \cite{Spohn80}, and the dynamics of the system is given by the interaction of the tagged particle with the background gas. In this model the obstacles do not interact among each other, differently from what happens in the non-ideal Rayleigh gas. Moreover, the energy of the tagged
particle in the Boltzmann-Rayleigh gas is not constant and, in particular, the expected time to the
next collision, 
changes during the evolution.  Equation \eqref{A0_0}-\eqref{A0_1} has been derived from Rayleigh gas in \cite{BLLS} and \cite{Spohn80}, and recently under more general assumptions in \cite{MST}, see also \cite{LNW} and \cite{NVW}. 
 We further remark that the linear collision operator $\Ll$ we are considering is related to the standard nonlinear Boltzmann collision operator $Q$ by means of the relation $\Ll(g)=Q(M,g)$. 

In this paper we will only consider kernels satisfying the so-called Grad's cut-off assumption i.e.
\begin{equation}\label{eq:GradCutOff}
    \int_{S^2} B(n \cdot \omega, |w|) d \omega <+ \infty \quad \text{for every $w \in \R$}
\end{equation}
and we will assume that $B$ is positive and homogeneous of parameter $\gamma \in \mathbb{R}$ i.e.
\begin{equation}\label{eq:homB}
    B(n \cdot \omega,\lambda|w|)=\lambda^{\gamma}B(n \cdot \omega, |w|) \quad \text{for all $\lambda>0,\, w \in \R$.}
\end{equation}
Moreover we will assume without loss of generality the following decomposition of $B$, namely
\begin{equation}\label{eq:decB}
    B(n \cdot \omega,|w|)=b(n\cdot \omega)|w|^{\gamma} \quad b \in L^{\infty}(S^2), b   \geq 0, \, \gamma \in \mathbb{R}.
\end{equation}

We are interested in studying a particular class of solutions to the Boltzmann-Rayleigh Equation called \emph{homoenergetic solutions.} These  solutions, which provide
insight about properties of Boltzmann gases in open systems, have been originally introduced in the nonlinear setting by 
 by Galkin \cite{Galkin1} and Truesdell
\cite{T} and later considered
both in the physical and mathematical literature. For instance, we refer to the books \cite{garzo} and \cite{TM}.  
From the mathematical side, the well-posedness of homoenergetic solutions to the nonlinear Boltzmann equation has been originally considered in \cite{CercArchive}, \cite{Cerc2000}
and more recently  
in \cite{JNV1}. The rationale to consider homoenergetic solutions to the Boltzmann equation is also motivated by an invariant manifold of solutions of the classical molecular dynamics with certain symmetry property, see \cite{DJ,DJ1} and see also \cite{JQW}.
We emphasize that the Boltzmann equation for homoenergetic flows can be thought as a model for dilute gases with boundary conditions at infinity yielding mechanical deformations. Therefore, it is not possible to obtain a detailed balance condition
associated to the  equation and the model can be thought as an  open system, and hence no thermalization to a stationary
state for long times can be expected. 
The long-time asymptotic behavior of homoenergetic solutions  
has been studied in
a series of papers \cite{BNV,JNV1,JNV2,JNV3,NV} and also in \cite{DL1, DL2, K1,K2}.

The goal of this paper is the analysis of homoenergetic solutions for the Rayleigh-Boltzmann equation, which have  been not yet considered in the mathematical literature. As discussed below in more detail, it is worth studying this class of solutions in the linear setting    since they exhibit 
new features both from the mathematical and physical point of view.  More precisely, the homoenergetic solutions we consider here are solutions to \eqref{A0_0}-\eqref{A0_1} having
the form
\begin{equation}
	g\left(  t,x,w\right)  =f\left(  t,v\right)  \text{   with }v=w-\xi\left(
	t,x\right).  \label{B1_0}%
\end{equation}
It is important to point out that this transformation does not act solely on the tagged particle, but on the full system consisting of tagged particle and background gas. 
This means that also   the Maxwellian distribution for the velocities of the background particles  is affected by the field $\xi(t,x)$. More precisely,  as a consequence of  \eqref{B1_0},   the background velocities are transformed according to  $v_* = w_*-\xi(t,x) $  and thus $M(w_\ast)$ is transformed into
\begin{equation}\label{B1_1}
\widetilde{M}(v_\ast)=(2\pi)^{- \frac{3}{2}}\exp\left(-\frac{|v_*|^2}{2}\right) .
\end{equation}
Notice that the requirement that the background is also affected by effect of the  deformation  makes the homoenergetic ansatz  physically consistent in the linear setting as well, see Appendix \ref{appA} for further details.

We observe that, under suitable integrability conditions, every solution of
(\ref{A0_0}) with the form (\ref{B1_0}) yields only time-dependent density $ \rho
	\left(  t,x\right)  =\rho\left(  t\right)$ and internal energy $\varepsilon\left(  t,x\right)  =\varepsilon\left(  t\right)$ with
\begin{equation}
\rho\left(t\right)=\int_{\mathbb{R}^3} f(t,v) dv , \quad  \varepsilon\left(t\right)= \frac 1 {2\rho\left(t\right)}
   \int_{\mathbb{R}^{3}}f\left(  t,v\right)  \vert v\vert^{2} dv .   \label{eq:Hom1}
\end{equation}
However, for the average velocity $V(t,x)$ one has $V\left(  t,x\right)
=\xi\left(  t,x\right)+W(t)$ with
\begin{equation}
    W(t)=
    \frac{1}{\rho(t)}\int v f(t,v) dv \ .
\end{equation}

For the linear Boltzmann equation instead the only relevant macroscopic quantity is $\rho(t)$ and therefore for homoenergetic solutions the spatial dependence disappears.

A direct computation shows that in order to have solutions of (\ref{A0_0})
with the form (\ref{B1_0}) for a sufficiently large class of initial data it must hold
\begin{equation}
	\frac{\partial\xi_j}{\partial x_k} \; \; \text{independent on $x$ and \; 
	$\partial_{t}\xi+\xi\cdot\nabla\xi=0$} .\label{B2_0}%
\end{equation}

The first condition implies that $\xi$ is an affine function on $x$ i.e. $\xi(t,x)=L(t)x+b(t)$ while the second one implies
\begin{align}
    \frac{d L(t)}{dt}+(L(t))^2 & = 0; \label{B4_0} \\
    \frac{d b(t)}{dt}+L(t) b(t) & = 0.
\end{align}
Without loss of generality it is possible to assume that $b(t)=0$ by the change of variables $\tilde w =w-b$, see \cite{JNV3} Remark 2.2 for more details.
Hence for the $L(t)$ we have then
\begin{equation}
	\frac{d L\left(  t\right)  }{d t}+\left(  L\left(  t\right)  \right)
	^{2}=0, \quad L (0) = A\in M_{3}\left(  \mathbb{R}\right), \label{B3_0}%
\end{equation}
where we have added an initial condition. 
Classical results of ODE theory show that there is a unique solution of \eqref{B3_0}, which reads
\begin{equation}
	L\left(  t\right)  =\left(  I+tA\right)  ^{-1}A=A\left(  I+tA\right)
	^{-1}, \label{B7_0}%
\end{equation}
and is defined on the maximal interval of existence $[0,a)$, where $\det \left( I+tA \right)>0 $. 
Then, combining \eqref{B1_0} and \eqref{B4_0}, equation \eqref{A0_0} becomes 
\begin{equation}\label{eq:homBol}
    \partial_t f-L(t)v \cdot \partial_v f = \Ll(f).
\end{equation}
where the collision operator $\Ll$ is defined in \eqref{A0_1}.   
We are interested in the long-time asymptotics of the 
function $\xi\left(  t,x\right)  =L\left(  t\right)  x = A\left(  I+tA\right) x$.  
The key idea is to study the form of the matrix $L(t)$ as $t \to \infty$ in a particular orthonormal basis using the Jordan normal form for real $3\times 3$ matrices. We refer to \cite{JNV1} Theorem $3.1$ where the full classification of all the possible long-time asymptotics of the matrix $L(t)$ in the case $\det(I+tA)>0$ for all $t \geq 0$ has been obtained.

In this work we will focus on the paradigmatic case of \emph{simple shear deformation}. 
In the \emph{simple shear} case the matrix $L(t)$ has the form
\begin{equation}\label{eq:ShearMat}
L = \left( \begin{array}{ccc} 0 & K & 0 \\
	0 & 0 & 0  \\
    0 & 0 & 0 
\end{array} \right) \quad \text{with $K>0$}.
 \end{equation}
With this choice of the matrix $L$ the linear Boltzmann equation for homoenergetic flows reads
\begin{equation}\label{eq:Bshear}
	\partial_tf-K v_2 \partial_{v_1}f=\Ll(f).
\end{equation}
Our main interest is to study the long-time behavior of the solutions to \eqref{eq:Bshear}. Despite the linearity of the model, the analysis in the liner setting has both a mathematical and a physical interest due to the fact that  the interplay between the effect of the shear deformation on the system and the collision dynamics of the tagged particle with the background thermal bath produces new features, different from the one arising in the nonlinear setting as will be pointed out below.  
In order to study the asymptotic behaviour of the solutions, as a first step, we 
look at the physical dimensions of the three terms appearing in the equation \eqref{eq:Bshear}, namely
\begin{equation}\label{eq:eqMagitude}
    \frac{[f]}{[t]},[f],[v]^{\gamma}[f]
\end{equation} 
where $[v]$ is the order of magnitude of $v$, $\gamma$ is the homogeneity of the collision kernel $B$,  as in \eqref{eq:homB}, and $[f]$ is the order of magnitude of $f$. 

The long-time behavior of the solution of \eqref{eq:Bshear} strongly depends on the homogeneity of the kernel $B$. For large times, one could expect three different possibilities:  either the collision term
becomes much larger than the transport term 
for large velocities, or the transport term is
much larger, or the transport term and the collision term are of the same order of magnitude. According to the terminology introduced in the nonlinear setting (e.g., cf. \cite{JNV1}-\cite{JNV3}) one refers to the first case as \emph{collision dominated regime}, that takes place when $\gamma>0$, and to  the second case as \emph{hyperbolic dominated regime}, when $\gamma<0$.  The only case in which the three terms balance is when $\gamma=0$,  which corresponds to the  \emph{Maxwellian molecules} case.  It is worth noticing that the values of $\gamma$ corresponding to each regimes depend on the deformation chosen and do vary for any possible choice of $L(t)$. 

In this paper we will consider the long-time behaviour of solutions to \eqref{eq:Bshear} in the case in which the collision kernel $B$ has homogeneity $0 \leq \gamma < 1$, hence  covering both the collision-dominated and the balance regime.  The novelty is that, for this range of values of $\gamma$, we will show the existence of a steady state. This steady state  is a non-equilibrium stationary solution, due to the presence of the shear term  that brings the system out-of-equilibrium. On physical grounds, its existence can be expected as a result of a   compensation between the effect of the shear deformation and the interaction of the tagged particle with the thermal bath, something that does not take place instead in the nonlinear setting when particles interact between each others.  

 More precisely, we will first focus on the case of  collision kernels $B$ which are homogeneous functions of degree zero with respect to $w$, that is kernels $B$ with the form
\begin{equation}\label{eq:AssB1}
    B(n \cdot \omega,|v|)=b(n \cdot \omega) \quad \text{with $b \in L^{\infty}(S^2)$}. 
\end{equation} 
We will then extend our analysis to the case $0< \gamma < 1$. As a first step we prove a well-posedness result for solutions to \eqref{eq:Bshear} in the setting of non-negative Radon measures,
for $\gamma \in [0,1)$. We treat the case $\gamma=0$ and $\gamma \in (0,1)$ separetely, in Section \ref{sec:3.1} and Section \ref{ssec:2.3} respectively. 
Building on these results, we address the main goal of the paper which is to prove the existence of a stationary solution to \eqref{eq:Bshear}, in the case $\gamma\in [0,1)$. In the case $\gamma=0$ this is performed by a careful analysis of the moments matrix $M(t)$ whose components are $M_{jk}(t)=\int_{\R} v_jv_kf(t,dv),$ with $j,k=1,2,3$. On the other hand for $0<\gamma<1$ the result is obtained by estimates on the second order moment. This will be the content of Section \ref{sec:3}  and Section \ref{ssec:statgamma>0} for the cases $\gamma=0$ and $\gamma\in (0,1)$, respectively. It is important to notice that the existence of a stationary solution in the regime $0<\gamma<1$ does not require any smallness condition on the shear deformation.  On the other hand, when $\gamma=0$, the existence of a stationary solution will be obtained under the assumption that $0 \leq K <l_0$,  with $l_0>0$ small, which otherwise would not give moments conservation. The existence of a stationary state is something that differs greatly from what happens in the nonlinear case. Indeed, for the nonlinear Boltzmann Equation in the case of simple shear deformations, in the case of pseudo Maxwell molecules, i.e. $\gamma=0$, under an analogous smallness condition on $K$, it has been possible to prove that there exists a unique self-similar profile which is asymptotically stable, see \cite{JNV1} and \cite{BNV}.  It is worth to  emphasize also a further novelty  of this model. More precisely, for the Rayleigh-Boltzmann equation, if the shear parameter $K$ is large, namely in the regime in which  $K>l_0$, it is possible to show that a steady state for the Rayleigh-Boltzmann equation with shear cannot exist due to the fact that the moments of order $s$ with $s>2$ increase at most exponentially in time (cf. Proposition \ref{prop:expMoments}). We conjecture that in this regime the behavior of the solution for long-times would be given by a self-similar solution to the following equation
\begin{equation}\label{eq:approxVlarge}
    \partial_tf-Kv_2\partial_{v_1}f=\int_{S^2}b(n\cdot \omega) (f(v-(v \cdot \omega)\omega)-f(v))d \omega
\end{equation}
which provides an approximation of \eqref{eq:homB}  for velocities $|v| \gg 1$. This approximation is formal, but will be object of further studies by the authors in the future. \\

\subsection*{Notation:} 
In this paper we will consider the space of non-negative and finite Radon measures on $\R$ that we will denote as $\mathscr{M}_+(\R)$. Moreover, given $s>0$, we will denote as $\mathscr{M}_{+,s}(\R)$ the subspace of $\M$ consisting of all measures $g$ satisfying
$$\int_{\R}|v|^sg(dv) < + \infty .$$
We will also use the following norm for $g \in \mathscr{M}_{+,s}$
\begin{equation}
    \|g\|_{\mathscr{M}_{+,s}}= \sup_{\vp \in C_0(\R), \|\vp\|_{\infty}=1} \int_{\R}(1+|v|)^s\vp(s) g(dv).
\end{equation}
Given a linear operator $(T,\mathscr{D}(T))$ we will denote with $\mathscr{D}(T)$ its domain and with $\mathcal{R}(T)$ its range. Moreover the closure of $(T,\D(T)))$ will be denoted as $(\bar T, \D(\bar T)))$. The matrix $I$ will denote the $3 \times 3$ identity matrix. Furthermore, we will use the sup norm for matrices and we will write
\begin{equation}\notag
    \|A\|_{\infty}=\sup_{i,j=1,2,3} |a_{ij}|,  \quad A=(a_{ij}) \in \mathrm{M}_3(\mathbb{R}).
\end{equation}
Finally we will always implicitly intend that the constants $C \in \mathbb{R}_+$ appearing in the proofs of this paper can change from line to line along the computations.

\section{Well-Posedness theory for the Boltzmann-Rayleigh equation with simple shear deformations. }
\label{sec:2}

In this section we study the well-posedness of the following Cauchy problem
\begin{equation}\label{eq:Cauchy2}
\begin{cases} 
	\partial_tf-K v_2 \partial_{v_1}f=\Ll(f), & \\
	f(0,v)=f_0(v). 
\end{cases}
\end{equation}
Notice that for cut-off kernels, see \eqref{eq:GradCutOff}, it is possible to decompose $\Ll(f)$ as
\begin{equation}\label{eq:decLl}
\Ll(f)=\Ll_+(f)-\Ll_-(f)=\left(\int_{\R}dv_* \int_{S^2}  B( n\cdot \omega,|v-v_*|) \, M_*'f'd \omega \,\right) -\nu(|v|)  f
\end{equation}
where $\nu(|v|)$ represents the scattering rate, namely 
\begin{equation}\label{eq:collfreqnu}
 \nu(|v|)=\int_{\R}dv_* \int_{S^2} \, B( n\cdot \omega,|v-v_*|) M_* d\omega>0.
\end{equation}
To obtain the weak formulation for the Cauchy problem \eqref{eq:Cauchy2} we multiply \eqref{eq:Cauchy2} by a smooth and compactly supported test function and integrate over $\R$ thus getting 
\begin{gather}
    \partial_t \left(\int_{\R} f(t,dv)\vp(t,v)\right)-\int_{\R}f(t,dv)\partial_t\vp(t,v)-\int_{\R}\left(Kv_2\partial_{v_1}f(t,dv)\right)\vp(t,v)=\int_{\R}\left(\Ll f(t,dv)\right)\vp (t,v).
\end{gather}
For the left hand side we get
\begin{gather}
    \partial_t \left(\int_{\R} f(t,dv)\vp(t,v)\right)-\int_{\R}f(t,dv)\partial_t\vp(t,v)+\int_{\R}f(t,dv)Kv_2\partial_{v_1}\vp(t,v)
\end{gather}
while for the right hand side, using Fubini and the collision rule we obtain
\begin{gather}
\int_{\R}\left(\Ll f\right)(t,dv)\vp (t,v)= \int_{\R}f(t,dv) \int_{\R} dv_* \int_{S^2}B(n\cdot \omega,|v-v_*|)M_*(\vp(v')-\vp(v)) d \omega\notag  \notag \\
= \int_{\R}f(t,dv)  (\Ll^*\vp)(t,v).
\end{gather}
If we further integrate with respect to $t$ over $[0,T]$ we get
\begin{gather}
    \int_{\R} f(T,dv)\vp(T,v)- \int_{\R} f(0,dv)\vp(0,v)   \notag \\ 
    = \int_0^T dt \int_{\R}f(t,dv)\partial_t\vp(t,v)\notag \ -\int_0^T dt \int_{\R}f(t,dv)Kv_2\partial_{v_1}\vp(t,v)+\int_0^T dt \int_{\R}f(t,dv)\Ll^*(\vp)(t,v).
\end{gather}

These formal identities motivate the following definition of weak solutions to \eqref{eq:Cauchy2}. 
\begin{defn}\label{def:weakSol}
 	Let $T>0$. A measure $f \in C\left( [0,+\infty),\M\right)$ is a \emph{weak solution} of   \eqref{eq:Cauchy2} if for every test function $\vp \in C^1\left( [0,+\infty),C_c^{\infty}(\R) \right) $ one has
 \begin{equation}
     \label{eq:weakf}
		\int_{\R}\vp(T,v)f(T,d v)- \int_{\R}\vp(0,v)f_0(dv) = \int_0^T dt \int_{\R} f(t,dv)\left[\partial_t\vp - K v_2\partial_{v_1}\vp+\Ll^*(\vp) \right](v,t)dt 
	\end{equation}
where $\Ll^*$ denotes the adjoint operator of $\Ll$, i.e.
\begin{equation}\label{eq:adjL}
\left(\Ll^* \vp\right)(v)=\int_{\R} \int_{S^2}B(n\cdot \omega,|v-v_*|)M_*(\vp'-\vp)dv_* d \omega.
\end{equation}
\end{defn}

\begin{remark}\label{rem:weakForm}
We point out that even though in the definition we required the test functions to be in $C^{\infty}_c(\R)$, equation \eqref{eq:weakf} is well-defined as soon as the test function $\vp$ is in $C_0(\R)$ and it has one partial derivative with respect to $v_1$ decaying faster that $\frac{1}{|v|}$.
\end{remark}

\subsection{Well-Posedness theory in the case of pseudo Maxwellian molecules (i.e. $\gamma=0$)}  \label{sec:3.1}

We focus here on the case 
of  pseudo-Maxwellian interactions, namely the case in which the collision kernel $B$  satisfies condition \eqref{eq:AssB1}, i.e. $B(n \cdot \omega,|v|)=b(n \cdot \omega)$ with $b \in L^{\infty}(S^2)$. 
Notice that under this assumption the scattering rate   defined in \eqref{eq:collfreqnu} is a positive constant, i.e. $\nu(|v|) = \nu>0$. 
The main result of this section is the following.
\begin{thm}\label{thm:wp1}
    Suppose that $f_0 \in \M$ and that $B$ satisfies \eqref{eq:AssB1}. Then, there exists a unique weak solution $f \in C\left( [0,+\infty),\M \right)$ of \eqref{eq:Cauchy2}  in the sense of Definition \ref{def:weakSol}.
\end{thm}

To prove Theorem \ref{thm:wp1} we will work with the adjoint equation of \eqref{eq:Cauchy2}, relying  on the duality provided by the classical Riesz-Markov-Kakutani Theorem. More precisely, motivated by the weak formulation in Definition \ref{def:weakSol}, we define a \emph{backward in time} dual adjoint problem for \eqref{eq:Cauchy2} which reads as  
\begin{equation}\label{eq:adjEq}
	-\partial_t  {\vp}+ Kv_2 \partial_{v_1} {\vp} = \Ll^*({\vp}),
 \end{equation}
 with final condition
 \begin{equation}\label{eq:adjEqT}
 \vp(T,v)=\vp_T(v), \quad \vp_T \in C^{\infty}_c(\R), T>0. 
\end{equation} 

In order to 
obtain a \emph{forward in time} formulation of the equation, 
we perform the change of variables $t\to s=T-t$, with $s \in [0,T]$ for $T>0$, 
which gives $\partial_s\vp+Kv_2\partial_{v_1}\vp-\Ll^*(\vp)=0$. Using the fact the variables are dummy variables, in what follows we will denote the new variable $s$ as $t$. This allows us to define the Cauchy problem for the \emph{forward in time} adjoint equation of \eqref{eq:homBol} as
   \begin{equation}\label{eq:adjCauchy2}
	\begin{cases}
		\partial_t\vp+ Kv_2 \partial_{v_1} \vp = \Ll^*(\vp), & \\
		\vp(0,v)=\vp_0(v) \; 
	\end{cases}
\end{equation}
with $\vp_0 \in C^{\infty}_c(\R)$.

To prove Theorem \ref{thm:wp1}, we will make use of the machinery of Markov semigroups, which allows to relate the existence of a unique solutions to the Cauchy problem \eqref{eq:adjCauchy2} to the fact that a suitable linear operator $(A,\mathscr{D}(A))$ generates a Markov  semigroup. To be more precise, define the linear operator $(A,\mathscr{D}(A))$  as
\begin{equation}\label{def:opA}
   A:\D(A) \subseteq C_0(\R) \rightarrow C_0(\R), \qquad  A\vp:=-Kv_2 \partial_{v_1}\vp + \Ll^*(\vp) \
\end{equation}
with domain
 \begin{equation}\label{def:DomA}
    \D(A)=\left \{\vp \in C_0(\R) \; \Big| \; \exists \  \partial_{v_1}\vp \in C_0(\R) \; \text{such that }  |\partial_{v_1} \vp| \leq \frac{C}{(1+|v|)^{1+\delta}},  C,\delta>0 \right\}.
\end{equation}
First notice that is possible to decompose $\Ll^*$ in a similar way as $\Ll$, namely
\begin{equation}
\label{eq:decLstar}   \Ll^*(\vp)=\Ll^*_+(\vp)-\Ll_-(\vp)=\int_{\R}\int_{S^2}B( n\cdot \omega,|v-v_*|) M_*\vp' dv_* d\omega-\nu(|v|)\vp.
\end{equation}
The operator $\Ll_+^*$ maps $C_0(\R)$ into itself. To see this, we observe that the space of Schwarz functions $\mathscr{S}(\R)$ is contained in the domain of $\Ll_+^*$, see Lemma \ref{lem:decayL*}, thus $\D(\Ll_+^*)$ is dense in $C_0(\R)$. Moreover $\|\Ll_+^*(\vp)\| \nu \|\vp\|$ for every $\vp \in \D(\Ll^*_+)$. Hence $\Ll_+^*$ is a closed and bounded, thus by the closed graph theorem its domain is closed which implies $D(\Ll_+^*)=C_0(\R)$. From this and from the choice of $\D(A)$ it follows that the operator $(A,\D(A))$ is well-defined on $C_0(\R)$.  
\medskip

We now recall some useful definitions on Markov generators and Markov semigroups on locally compact spaces. 

\begin{defn}\label{def:preGen} 
    Let $\Omega$ be locally compact. A linear operator $T: \D \mathscr(T) \subseteq C_0(\Omega) \rightarrow C_0(\Omega)$ is a Markov pre-generator if 
    \begin{itemize}
        \item [$(i)$] $\mathscr{D}(T)$ is dense in $C_0(\Omega)$;
        \item [$(ii)$] For every sequence $(\vp_n)_n \subseteq \D(T)$ such that $\vp_n \rightarrow 1$ pointwise one has $T\vp_n \rightarrow 0$ pointwise;
        \item [$(iii)$] If $\vp \in \mathscr{D}(T)$ has positive maximum at $x_0$ then $(T\vp)(x_0) \leq 0$.
    \end{itemize}
\end{defn}

\begin{remark}
    It is well-known that in order to have a generation result for semigroup the operator under consideration must be closed. It can be shown that every Markov pre-generator $T$ has a closure $\overline{T}$ which is again a Markov pre-generator. Furthermore, condition $(iii)$ implies that if $A$ is dissipative and  if it generates a semigroup $S(t)$ then $S(t)$ is a positive semigroup.
\end{remark}

\begin{defn}\label{def:markovGen}
    A closed Markov pre-generator $(T,\D(T))$ is a Markov generator if the exists $\bar \lambda> 0$ such that $\mathcal{R}(\bar \lambda - T ) = C_0(\Omega)$ i.e. the resolvent $R(\lambda,T)$ is surjective, moreover then $\mathcal{R}(\lambda -T) = C_0(\Omega)$ holds for every $\lambda>0$.
\end{defn}

\begin{remark}\label{rem:closureGenerator}
    In most applications the operator $T$ is not closed, nor it is easily possible to determine its closure. In these cases, instead of proving that $\mathcal{R}(\lambda-A)=C_0(X)$, one has to prove the weaker condition $\mathcal{R}(\lambda-A)$ is dense in $C_0(X)$ for some $\lambda>0$ that is the equation
    \begin{equation}
        (\lambda-A)\vp=g \quad \vp \in D(T)
    \end{equation}
    has unique solution for $g$ in a dense subspace of $C_0(X)$ for some $\lambda>0$. From this it follows that $(\bar T,\D(\bar T))$ is Markov generator since the range of $\lambda-\bar T$ is dense and closed and hence it is the whole $C_0(\R)$. 
\end{remark}

\begin{defn}\label{def:markovSem}
    A Markov semigroup $S(t): C_0(\Omega) \rightarrow C_0(\Omega)$  is a family of linear operators such that
    \begin{itemize}
        \item [$(i)$] $S(0)=I$;
        \item [$(ii)$] $S(t+s)=S(t)S(s)$ for every $t,s \geq 0$;
        \item [$(iii)$] $\lim_{t \rightarrow 0^+} \|S(t)f-f\|=0$ for every $f \in C_0(\Omega)$;
        \item [$(iv)$] For every sequence $(\vp_n) \subseteq C_0(\Omega)$ such that $\vp_n \rightarrow 1$ pointwise one has $S(t)\vp_n\rightarrow 1$ pointwise;
        \item [$(v)$] $S(t)f \geq 0$ whenever $f \geq 0$.
    \end{itemize}
\end{defn}

\begin{thm}[Hille-Yosida]\label{thm:HY}
    There is a one-to-one correspondence between Markov generators and Markov semigroups, namely every Markov generator is the generator of a Markov semigroup.
\end{thm}

We now focus on the operator  $(A, \D(A))$ defined in \eqref{def:opA}. It is immediate to see that $(A, \D(A))$ is a Markov pre-generator.    
It is well-known that the closure $\bar A$ of $A$ generates a semigroup if and only if   \eqref{eq:adjCauchy2} has a unique solution $\vp \in C([0,+\infty),\D(A))$, see \cite{Pazy}. The semigroup generated by $(\bar A,\D(\bar A))$ is actually Markov since $\bar A$ is a Markov pre-generator and by the classical Lumer-Philips theorem one has $\mathcal{R}(\lambda - \bar A)=C_0(\R)$ for all $ \lambda >0$.


For this reason we will now provide a well-posedness result for the adjoint problem \eqref{eq:adjCauchy2}. 
The adjoint Cauchy problem then reads
\begin{equation}\label{eq:AdjCauchy}
\begin{cases}
    \partial_t\vp + Kv_2 \partial_{v_1} \vp = \Ll_+^*(\vp)-\nu\vp & \\
    \vp(0,v)=\vp_0(v). 
\end{cases} 
\end{equation}
Using Duhamel's formula, from \eqref{eq:AdjCauchy} we have 
\begin{equation}\label{DuhamelAdj}
   \vp(v,t)=e^{-\nu t}\vp_0(\gamma(t))+\int_0^t e^{-\nu(t-\tau)}(\Ll_+^*\vp)(\gamma^t(v)) d\tau
\end{equation}
where $\gamma^t(v)$ is the characteristic flow 
\begin{equation}
    \gamma^t(v)=(v_1- Kv_2 t ,v_2,v_3), \quad v \in \R.
\end{equation}

Equivalently, introducing  the free-streaming operator $T(t)$, namely the the semigroup generated by the operator $B\vp=-\nu\vp-Kv_2\partial_{v_1}\vp$, defined by
\begin{equation}\label{eq:freeStreaming}
    T(t)\vp(v)=\exp\left(-\nu t \right)\vp(\gamma^t(v)) \quad \vp \in C_0(\R)
\end{equation} 
it is possible to write \eqref{DuhamelAdj} as
\begin{equation}\label{DuhamelT}
    \vp(t,v)=T(t)\vp_0(v)+\int_0^t T(t-\tau)\Ll^*_+(\vp)(v,\tau)d\tau \quad t \in [0,T].
\end{equation}

\begin{prop} \label{thm:wp2}
    Let $\vp_0 \in C_0(\R)$ and let $B$ satisfy assumption \eqref{eq:AssB1}. Then there exists a unique $\vp \in C([0,+\infty),C_0(\R))$ solution to the initial value problem \eqref{eq:adjCauchy2}.
\end{prop}

\begin{proof}[Proof of Proposition  \ref{thm:wp2}]
    Let $T>0$. Define the operator $\mathscr{S}_T:C([0,T],C_0(\R))\rightarrow C([0,T],C_0(\R))$ by means of 
    $$\mathscr{S}_T[\vp](t,v)=T(t)\vp_0(v)+\int_0^t T(t-\tau)(\Ll_+^* \vp)(v,\tau)d\tau$$
    for $t \in [0,T]$.
    Therefore, due to   \eqref{DuhamelAdj}, the proof of the Proposition reduces to proving existence and uniqueness of solutions for the fixed point equation $\vp=\mathscr{S}_T[\vp]$. 

    Define the space 
    \begin{equation}
        Y=\left\{ \vp \in C([0,T],C_0(\R))\; \Big | \; \sup_{t \in [0,T]} \|\vp(t)\|_{\infty} \leq 2\|\vp_0\| \right\}.
    \end{equation}
    The operator $\mathscr{S}_T$ maps the space $Y$ onto itself for a sufficiently small $T$. In fact one has
    \begin{equation}
        \|\mathscr{S}_T[\vp] (t) \| \leq \|\vp_0\|+cT\|\vp \| \leq (1+2cT)\|\vp_0\|
    \end{equation}
    where the boundedness of the operator $\Ll_+^*$ and the fact that $\|T(t)\| \leq 1$ were used. Hence if $T <1/2c$ taking the supremum over $[0,T]$ gives $\mathscr{S}_T(Y) \subseteq Y$. Notice that by Gronwall Lemma the following estimate holds
    \begin{equation}
        \|\vp(t) \| \leq \|\vp_0\|\exp(2cT).
    \end{equation}
    Furthermore the map $\mathscr{S}_T$ is a contraction on $Y$ since by out choice of $T$ we have
    \begin{equation}
        \sup_{t\in [0,T]}\|(\mathscr{S}_T[\vp_1-\vp_2)](t)\| \leq cT\sup_{t \in [0.T]}\|(\vp_1-\vp_2)(t)\|< \sup_{t \in [0.T]}\|(\vp_1-\vp_2) (t)\|.
    \end{equation}
    Thus by Banach fixed point theorem there exists a unique $\vp \in C([0,T],C_0(\R))$ such that 
    \begin{equation}
        \vp(v,t)=e^{-\nu t}\vp_0(v)+\int_0^te^{-\nu (t-\tau)}(\Ll_+^* \vp)(v,\tau) d \tau.
    \end{equation}
    By the choice of $Y$ it follows that $\| \vp(t)\|$ remains bounded for all times in $[0,T]$, hence, repeating the same argument, it is possible to extended the solution $f$ to an interval $[T,T+\delta]$ for some $\delta>0$. Iterating this procedure we obtain that $\vp \in C([0,+\infty),C_0(\R))$.   
\end{proof}

    By a standard result on semigroup theory we know that $\vp(t,v) \in C^1([0,T],C_0(\R))$ if and only if $\vp_0 \in \D(A)$. Moreover, the well-posedness proven in Theorem \ref{thm:wp2} is equivalent to the fact that $(\bar A, \mathscr{D}(\bar A))$ generates a strongly continuous semigroup $S(t)$ acting as
    \begin{equation}\label{eq:semiS}
        S(t): C_0(\R) \rightarrow C_0(\R) \quad S(t)\vp=T(t)\vp_0(v)+\int_0^tT(t-\tau)(\Ll_+^* \vp)(\tau) d \tau
    \end{equation}
    where $T(t)$ is defined in \eqref{eq:freeStreaming}, see \cite{Pazy}.

\medskip


\medskip

We can now conclude the proof of the well-posedness theorem for \eqref{eq:Cauchy2} stated above. 

\begin{proof}[Proof of Theorem \ref{thm:wp1}] For the proof of the existence of a weak solution, we construct a measure-valued function $f \in C\left([0, T ]; \mathscr{M}_+(\R)\right)$ using the unique solution
to the adjoint problem. 
    By Proposition \ref{thm:wp2} the operator $\bar A$ generates a Markov semigroup $S(t)$, furthermore by Remark \ref{rem:weakForm} if $\vp_0 \in \D(A)$ then $\vp(t,v)=S(t)\vp_0$ can be chosen ad test function in Definition \ref{def:weakSol}. Moreover, the existence of the semigroup $S(t)$ is equivalent to the global well-posedness in time of the \emph{forward} adjoint Cauchy problem \eqref{eq:adjCauchy2}. Now for every $f_0 \in \M$  and $\vp \in C^1([0,+\infty),C_0(\R))$ we define $f(t,dv) \in C([0,+\infty),\M)$ by means of the \emph{duality formula} as 
    \begin{equation}\label{eq:dualityFormula}
        \int_{\R}\vp(v)f(t,dv)=\int_{\R}\vp(v)S^*(t)f_0(dv)=\int_{\R} S(t)\vp(v) f_0(dv). 
    \end{equation}
    Due to existence of the semigroup this definition is well-posed.
    Let now $u(t) \in C^1([0,+\infty),\mathscr{D}(A))$. Notice that $\bar Au(t)=Au(t)$, moreover $S(t)(Ax)=A(S(t)x)$ if $x\in \mathscr{D}(A)$.
    Set 
    \begin{equation}
        \xi(t,v)=\partial_t u(t,v)+Au(t,v).
    \end{equation}
    One has that
    \begin{equation}\label{xi}
        S(t)\xi(t,v)=\partial_t(S(t)u(t,v)).
    \end{equation} 
Applying $f_0 \in \M$ to \eqref{xi} and integrating over $[0,T]$ with respect to $t$ yields
\begin{equation}
    \int_0^Tdt\int_{\R}f_0(dv)S(t)\xi(t,v)= \int_0^Tdt\int_{\R}f_0(dv)\partial_t(S(t)u(t,v))=\int_{\R}f_0(dv)(S(T)u(T,v)-u_0(v)).
\end{equation}
By the duality formula we get
\begin{equation}
    \langle f(T),u(T)\rangle-\langle f_0,u(0)\rangle=\int_0^T\langle f_0,S(t)\xi(t) \rangle=\int_0^T\langle f(t),\xi(t) \rangle= \int_0^T \langle f(t),\partial_t u(t)+Au(t) \rangle
\end{equation}
which proves that $f(t,dv)$ is a weak solution.

To prove uniqueness, suppose that there are two weak solutions $f_1,f_2$, in the sense of Definition \ref{def:weakSol}, 
to \eqref{eq:adjCauchy2} with the same initial datum $f_0$. We set $g=f_1-f_2$. Notice that $g$ also is a solution to \eqref{eq:adjCauchy2}, in the sense of Definition \ref{def:weakSol}, with  initial data $g_0\equiv 0$. 
Let $u$ be a non-zero test function such that
\begin{equation}
    \int_{\R}g(T,d v)u(T,v) \neq 0 \; \text{for some $T$}.
\end{equation}
Then the function $u(t)=S(T-t)u(T,v)$ is the unique solution of 
\begin{equation}\label{cauchyAux}
\begin{cases}
	-\frac{d u}{ d t}(t)=-Kv_2\partial_{v_1}  u(t)+\Ll^* u(t) & \\
	u(T)=\varphi(T) \; \;  \;\text{for $0 \leq t\leq T$.}
\end{cases}
\end{equation}
Equation \eqref{eq:weakf} then yields
\begin{gather}
    \int_{\R}g(T,d v)u(T,v)= 
    \int_0^T \int_{\R}g(t,d v)\left[\partial_t  u(t,v)-Kv_2\partial_{v_1} u(t)+\Ll^*u(t)\right]dt.
\end{gather}
The right-hand side is identically zero by \eqref{eq:weakf}, hence it must hold $g(t,v)\equiv 0$ for all $t$. A contradiction. 
\end{proof}

\begin{remark}
    It is worth noticing that the well-posedness proof given in this section is still valid if one assumes $\gamma \leq 0$ in \eqref{eq:homB}, the case $\gamma=0$ being the one treated previously. This is true because it is possible to prove that the scattering rate $\nu(|v|)$ defined as in \eqref{eq:collfreqnu} satisfies
    \begin{equation}
        C_1(1+|v|)^{\gamma} \leq \nu(|v|) \leq C_2(1+|v|)^{\gamma} \quad \text{for $C_1,C_2>0, \gamma \in \mathbb{R}$}
    \end{equation}
    from which it follows that $\nu(|v|)$ is bounded if $\gamma \leq 0$. The only difference between the case $\gamma=0$ and $\gamma<0$ is that in the latter the collision frequency $\nu$ still depends on $v$ which leads to a slight modification in the proof of Proposition \ref{thm:wp2} and of Theorem \ref{thm:wp2}.
\end{remark}

\subsubsection{A comment on the well-posedness result for \eqref{eq:Cauchy2} by means of Fourier transform} \label{ssec:2.}

As an additional remark we discuss here a different approach to prove the well-posedness result for the Cauchy problem \eqref{eq:Cauchy2} (e.g. Theorem \ref{thm:wp1}). This approach relies on the well developed machinery available for the study of the Boltzmann equations in the case of Maxwell molecules by means of the Fourier transform, a method which was introduced originally by A. Bobylev in the nonlinear spatially homogeneous setting and recently used in \cite{BNV} to study homoenergetic solutions of the nonlinear Boltzmann equation. More precisely, in  \cite{BNV} it has been possible to prove existence, uniqueness and stability of a self-similar profile for the Boltzmann equation for homoenergetic flows . 

We rewrite the linear  Rayleigh-Boltzmann equation \eqref{eq:Bshear} as follows
\begin{equation} \label{eq:GenBolt}
\partial_{t} f- \partial_{v} \cdot \left(L v f \right)    =\widetilde{\Ll}(f)
\end{equation}
where $L\in M_{3}(\mathbb{R})$ is the shear matrix defined in \eqref{eq:ShearMat} and with  
\begin{equation}\label{eq:CollBolt}
\widetilde{\Ll}(f)\left(  v\right)= \frac{1}{4 \pi}\int_{\mathbb{R}^{3}}dv_* \int_{S^{2}}  b\left(n \cdot \omega \right)\left(M_*'f'-M_*f\right) d\omega
\end{equation} 
We consider the initial value problem for \eqref{eq:GenBolt} with initial datum
\begin{equation}\label{eq:incdt}
f(0,v)=f_0(v)\geq 0, \qquad \int_{\R}  f_0(v)dv=1.
\end{equation} 
The normalization in \eqref{eq:incdt} can be assumed without any loss of generality.

\medskip

The great advantage of using the Fourier transform method consists in the fact that it allows to simplify the collision operator \eqref{eq:CollBolt}, see \cite{Bo75,Bo88}, but also the hyperbolic term $\mathrm{div}_v(Lf)$, which takes  a simpler form, see for instance \cite{BNV}. 
We denote by $\varphi(k)=\hat{f}(k)$ the Fourier Transform of the one particle probability density $f(v)$, i.e.
\begin{equation}\label{eq:phif}
\varphi(k)=\hat{f}(k)=\int_{\mathbb{R}^3}  f(v)e^{-i k\cdot v}dv, \;\; k\in\mathbb{R}^3.
\end{equation}
Using \eqref{eq:AssB1}, from \eqref{eq:GenBolt} we obtain the following equation
\begin{equation}\label{eq:BoltFourier}
\partial_t \varphi  + \left(Lk\right)\cdot \partial_k \varphi = \mathcal{I}^{+}(\widehat{M},\vp)-(\widehat{M})_{|_{k=0}} \vp 
\end{equation}
where 
\begin{equation} \label{eq:BoltCollFourier}
 \mathcal{I}^{+}(\widehat{M},\vp)(k)= \frac{1}{4 \pi} \int_{S^{2}}  b\left(  \hat{k} \cdot \omega \right)\varphi(k_{+}) \widehat{M} (k_{-})d\omega \;\; \text{with}\;\; k_{\pm}=\frac{1}{2}\left( k\pm \vert k \vert n\right),\quad \hat{k}=\frac{k}{\vert k\vert}.
\end{equation}
with $\widehat{M}(k)=e^{-\frac{|k|^2}{2}}$. The derivation of this formula can be found in \cite{Bo75, Bo88}. 
The initial condition \eqref{eq:incdt} becomes
\begin{equation} \label{eq:incdtFourier}
\varphi (0,k)=\varphi_0(k)= \int_{\R}  f_0(v)e^{-i k\cdot v}dv, \quad \varphi_0(0)=1.
\end{equation}
We notice that \eqref{eq:BoltFourier} implies the mass conservation condition $\varphi(t,0)=\varphi_0(0)=1$. 
Thus, equation \eqref{eq:BoltFourier} reads 
\begin{equation} \label{eq:BoltFourier2}
\partial_t \varphi  +\varphi + \left(Lk\right)\cdot \partial_k \varphi = \mathcal{I}^{+}(\widehat{M},\varphi).   
\end{equation}

We observe that in this section we can consider without loss of generality, rescaling the time unit if needed, collision kernels $b$ such that
\begin{equation}
    \int_{S^2}b( n \cdot \omega) d \omega = \bar b
\end{equation}
where $\bar b$ is such that $\mathcal{I}^+$ has supremum norm bounded by $1$, for more details we refer to \cite{BNV}, see proof of Lemma 3.1.

Following the same strategy as in \cite{BNV} one introduces a suitable space of functions for the Cauchy problem  \eqref{eq:incdtFourier}, \eqref{eq:BoltFourier2} which is the space of time-dependent probability measures $f(t,\cdot)$ in $\R$, $t \geq 0$. Then, it is natural to solve  \eqref{eq:incdtFourier}, \eqref{eq:BoltFourier2}  in the space of characteristic functions, i.e. the space of Fourier transforms $\vp (t,\cdot)$ of time-dependent probability measures $f(t,\cdot)$, $t \geq 0$, see for instance \cite{Fe2}. More precisely, we introduce the space \begin{equation}\label{def:setffi}
\Xi=\left \{\varphi\in C(\R,\mathbb{C})\; | \; \varphi(k)= \hat{f}(k), \text{for $ f\in \mathcal{P}_{+}(\R)$}  \right\} \subseteq C(\R;\mathbb{C}),
\end{equation}
where $\mathcal{P}_{+}(\R)$ denotes the set of Radon probability measures in $\R$ and the Fourier transform $\hat{f}(\cdot)$ is defined as in \eqref{eq:phif}. 
Moreover, given $T>0$ we define the space $C([0,T]; \Xi)$ with $\Xi$ endowed with the topology given by $\|\cdot\|_{\infty}.$ Furthermore we introduce a topology in $C([0,T]; \Xi)$ by means of the metric 
\begin{equation*} 
 d_{T}(\psi_1,\psi_2)=\sup_{0\leq t\leq T} \vert| \psi_1(\cdot,t) -\psi_2(\cdot,t) \vert|_{\infty}.
\end{equation*}
The space  $C([0,T]; \Xi)$ is a complete metric space, being closed in the topology of the uniform convergence.
It is then possible to prove the following well-posedness result. 
\begin{thm}\label{thm:1}
Let be $\vp_0\in \Xi$. Then there exists a unique $\vp \in C([0,\infty); \Xi)$ solution of \eqref{eq:BoltFourier2} with initial datum $\vp_0$. 
\end{thm}
In order to prove the above theorem one reformulates the Cauchy problem  \eqref{eq:incdtFourier}-\eqref{eq:BoltFourier2} as an integral equation using Duhamel’s formula and then resort to a standard Banach fixed point argument. We will not include here the details of the proof, since the structure of the proof follows the same lines of  \cite{BNV} where an
analogous result has been proven  for the nonlinear Boltzmann equation. For further details we refer to the same paper, Section $4$.

\subsection{Well-posedness theory in the case of collision kernels  with homogeneity $\gamma>0$}\label{ssec:2.3}
In this subsection we consider instead the case of cut-off collision kernels with $\gamma>0$ in \eqref{eq:decB}. 
The main difference with the case of homogeneity $\gamma \leq 0$ is that the scattering rate $\nu(|v|)$ defined in \eqref{eq:collfreqnu} is unbounded. The main theorem of this subsection is the following.
 \begin{thm}\label{thm:wpGamma>0}
    Suppose that $f_0 \in \M$ and that $B$ satisfies \eqref{eq:decB} with $\gamma\in (0,1)$. Then, there exists a unique weak solution $f \in C\left([0,+\infty),\M \right)$ of \eqref{eq:Cauchy2} in the sense of Definition \ref{def:weakSol}.
\end{thm}

    To prove the well-posedness of \eqref{eq:Cauchy2} for $\gamma \in (0,1)$ in the sense of Definition \ref{def:weakSol}, namely Theorem \ref{thm:wpGamma>0} above, we will rely on the same strategy used in Section \ref{sec:3.1}, i.e. we first prove the well-posedness for the adjoint Cauchy problem \eqref{eq:adjCauchy2} and then  prove the existence and uniqueness of a weak solution of \eqref{eq:adjCauchy2}. To this end, we first introduce the operator $A$ defined as $A\vp:=-Kv_2\partial_{v_1}\vp+\Ll^*(\vp)$ where $\Ll^*$ is the adjoint collision operator defined in \eqref{eq:adjL} and with domain
\begin{equation}
    \D(A)=\left\{ \vp \in C_0(\R) \; \Big | \; \exists \partial_{v_1} \vp \in C_0(\R) \; \text{for every $v \in \R$ and } \sup_{v \in \R} (1+|v|)^{1 + \gamma} \{|\vp(v)|,|\partial_{v_1} \vp |\} < + \infty \}\right\}.
\end{equation}
To see that $(A,\D(A))$ is well-defined on $C_0(\R)$ we have only to check that $\Ll^*_+(\vp) \in C_0(\R)$, being $\Ll^*_+$ the gain term of $\Ll^*$ as in \eqref{eq:decLstar}. This is the content of  the following Lemma.

\begin{lem}\label{lem:decayL*}
    Let the collision kernel $B$ be as in \eqref{eq:decB} with $\gamma \in [0,1)$ and let $\vp \in \D(A)$. Then $\Ll^*_+(\vp) \in C_0(\R)$.
\end{lem}

\begin{proof}
    From our assumptions on $B$ it immediate to see that $\Ll^*_+(\vp)$ is continuous. It only remains to prove that $\lim_{|v|\rightarrow + \infty} |\Ll^*_+(\vp)|=0$. Fix $\beta, \alpha\in \mathbb{R}$, with $\alpha>1+\gamma>0, \beta \in (0,1]$. We observe that  
    \begin{align}
        |\Ll^*_+(\vp)| & \leq \int_{\R} \int_{S^2}b(n\cdot \omega)|v-v_*|^{\gamma}M_*|\vp(v')|d\omega d v_* \notag \\ 
        & \leq   C \int_{\R} \int_{S^2}M_* \frac{|v-v_*|^{\gamma} }{(1+|v-((v-v_*)\cdot \omega)\omega|)^{\alpha}}d\omega d v_* \leq 
        C \int_{\R} \int_{S^2}\frac{|v|^{\gamma}+|v_*|^{\gamma}}{(1+|v-((v-v_*)\cdot \omega)\omega|)^{\alpha}} M_*d\omega d v_* \notag \\
        & \leq   C(1+|v|)^{\gamma}\int_{\R} \int_{S^2}\frac{\sqrt{M_*}}{(1+|v-((v-v_*)\cdot \omega)\omega|)^{\alpha}}d\omega d v_*=\mathcal{I},
    \end{align}
    where we used that $M_*(v) \leq C \sqrt{M_*(v)}$  and  $ |v_*|^{\gamma} M_*(v) \leq C \sqrt{M_*(v)}$. 
    Let now $v=|v|e$, where $e$ is a unit vector, then  we can write $\mathcal{I}$
\begin{gather}
        \mathcal{I} = C (1+|v|)^{\gamma}\left( \int_{\{|v_*|\geq |v|^{\beta}\}}\sqrt{M_*}dv_* \int_{S^2} \frac{d \omega}{\left(1+|v|\Big |e-(e\cdot \omega) \omega+\frac{(v_*\cdot \omega)}{|v|}\omega \Big |\right)^{\alpha}} \right. \notag \\
        \left. + \int_{\{|v|^{\beta}\geq |v_*|\}} \sqrt{M_*}dv_*\int_{S^2} \frac{d \omega}{\left(1+|v|\Big |e-(e\cdot \omega) \omega+\frac{(v_*\cdot \omega)}{|v|}\omega \Big |\right)^{\alpha}} \right)=\mathcal{I}_1+\mathcal{I}_2.
    \end{gather}
    Since the function
    \begin{equation}\notag
        \frac{1}{\left(1+|v|\Big |e-(e\cdot \omega) \omega+\frac{(v_*\cdot \omega)}{|v|}\omega \Big |\right)^{\alpha}}
    \end{equation}
    is bounded, the integral $\mathcal{I}_1$ can be readily estimated as
    \begin{equation}
        \mathcal{I}_1 \leq C (1+|v|)^{\gamma}e^{-C|v|^{2 \beta}} \leq C \frac{(1+|v|)^{\gamma}}{(1+|v|)^{\alpha}}
    \end{equation}
    which tends to $0$ as $|v|\rightarrow + \infty$. Estimating $\mathcal{I}_2$ is more involved, in fact, let $\theta$ be the angle between $e$ and $\omega$. Then changing into polar coordinates, turns $\mathcal{I}_2$ into
    \begin{gather}
      \mathcal{I}_2=  C(1+|v|)^{\gamma} \int_{\{|v|^\beta \geq |v_*|\}} \sqrt{M_*} dv_*\int_0^{2 \pi}d\vp \int_0^{\pi} \frac{\sin \theta d\theta}{\left(1+|v|\Big |e-(e\cdot \omega) \omega+\frac{(v_*\cdot \omega)}{|v|}\omega \Big |\right)^{\alpha}}
    \end{gather}
   where, with a slight abuse of notation, we have written $\omega$ for $\omega=\omega(\theta,\vp)$.
   Notice that $|e-(e\cdot \omega)\omega|=|1-\cos \theta|$ and notice that since $(e-(e \cdot \omega)\omega) \perp \omega$ we have that
    \begin{equation}\notag
        \Big |(e-(e \cdot \omega)\omega)+ \frac{v_* \cdot \omega}{|v|}\omega \Big |= \sqrt{|(e-(e \cdot \omega)\omega)|^2+\frac{|v_* \cdot \omega|^2}{|v|^2}}.
    \end{equation}
   For $\delta>0$ and sufficiently small we have that
    \begin{gather}
        \mathcal{I}_2 = C(1+|v|)^{\gamma} \left( \int_{\{|v|^\beta \geq |v_*|\}} \sqrt{M_*}dv_* \int_0^{2\pi}d\vp \int_0^{\delta} \frac{\sin \theta}{\left(1+\sqrt{|1-\cos \theta|^2+\frac{|v_* \cdot \omega|^2}{|v|^2}}\right)^{\alpha}}d\theta  \right. \notag \\
        \left. + \int_{\{|v|^ \beta \geq |v_*|\}} \sqrt{M_*}dv_* \int_0^{2\pi}d\vp \int_{\delta}^{\pi} \frac{\sin \theta}{\left(1+\sqrt{|1-\cos \theta|^2+\frac{|v_* \cdot \omega|^2}{|v|^2}}\right)^{\alpha}}d\theta \right) = \mathcal{I}_{21}+\mathcal{I}_{22}
    \end{gather}
    We first estimate $\mathcal{I}_{22}$. If $0<\delta \leq \theta \leq \pi$ we have that $|1-\cos \theta|\geq \sigma_{\delta}>0$, moreover $\Big |\frac{(v_* \cdot \omega)\omega}{|v|}\Big | \leq 1$ since $|v_*| \leq |v|^ \beta$. This implies that
    \begin{gather}
         \int_0^{2 \pi}d\vp \int_{\delta}^{\pi} \frac{\sin \theta d\theta}{\left(1+|v|\Big |e-(e\cdot \omega) \omega+\frac{(v_*\cdot \omega)}{|v|}\omega \Big |\right)^{\alpha}} \leq  \int_0^{2 \pi}d\vp \int_{\delta}^{\pi} \frac{\sin \theta d\theta}{\left(1+|v|\sigma_{\delta}\right)^{\alpha}} \leq \frac{C}{(1+|v|)^{\alpha}}
    \end{gather}
    which then implies
    \begin{equation}
        \mathcal{I}_{22} \leq C \frac{(1+|v|)^{\gamma}}{(1+|v|)^{\alpha}}.
    \end{equation}
    Now for $\mathcal{I}_{21}$ using the fact that $|1-\cos \theta| \approx \frac{\theta^2}{2}, |\sin \theta| \approx \theta$ if $0 \leq \theta \leq \delta$ we have that
    \begin{align}
         & \int_0^{2 \pi}d\vp \int_0^{\delta} \frac{\sin \theta d\theta}{\left(1+|v|\Big |e-(e\cdot \omega) \omega+\frac{(v_*\cdot \omega)}{|v|}\cdot \omega\omega \Big |\right)^{\alpha}}  \leq C \int_0^{2 \pi}d\vp \int_{0}^{\delta} \frac{\theta d\theta}{\left(1+|v|\sqrt{\frac{\theta^4}{4}+\frac{|v_* \cdot \omega|^2}{|v|^2}}\right)^{\alpha}}  \notag \\
         & =  \int_0^{2 \pi} \frac{|v_* \cdot \omega|}{|v|}d \vp \int_0^{\sqrt{\frac{|v|}{|v_* \cdot \omega|}} \; \delta} \frac{\xi}{\left(1+|v_* \cdot \omega|\sqrt{\xi^4+1}\right)^{\alpha}} d\xi     
    \end{align}
    where in the last line we have performed the change of variables $\sqrt{\frac{|v|}{|v_* \cdot \omega|}} \; \theta=\xi$. Since we have that
    \begin{align}
        & \int_0^{2 \pi} \frac{|v_* \cdot \omega|}{|v|} d \vp \int_0^{\sqrt{\frac{|v|}{|v_* \cdot \omega|}} \; \delta} \frac{\xi}{\left(1+|v_* \cdot \omega|\sqrt{\xi^4+1}\right)^{\alpha}} d\xi \leq 2\pi  \frac{|v_* \cdot \omega|}{|v|}  \int_0^{\sqrt{\frac{|v|}{|v_* \cdot \omega|}} \; \delta} \frac{\xi}{\left(1+|v_* \cdot \omega|\xi^2\right)^{\alpha}} d\xi  \notag \\
        &= \frac{\pi}{|v|(\alpha-1)}\left(1-\frac{1}{(1+|v|\delta^2)^{\alpha-1}}\right) \leq \frac{C_{\alpha}}{|v|}
    \end{align}
    it follows that
    \begin{gather}
        \mathcal{I}_{21} \leq C \frac{(1+|v|)^{\gamma}}{|v|} \int_{\{|v|^ \beta  \geq |v_*|\}}|v_*| \sqrt{M_*}dv_* \leq C \frac{(1+|v|)^{\gamma}}{|v|} \int_{\R}|v_*| \sqrt{M_*}dv_*= \frac{C}{|v|}(1+|v|)^{\gamma}.
    \end{gather}
    Finally choosing $\beta=1$ and from the assumption $\gamma \in (0,1)$ we get $\mathcal{I}_{21}\rightarrow 0$ as $|v| \rightarrow +\infty$. This concludes the proof.
\end{proof}

\bigskip

It is straightforward to see that $(A,\D(A))$ is a Markov pre-generator in the sense of Definition \ref{def:preGen}. From Remark \ref{rem:closureGenerator} to prove that the closure $\bar A$ of $A$ is a Markov generator we must prove that
\begin{equation}\notag
    \mathcal{R}(\lambda-A) \quad \text{is dense in $C_0(\R)$ for some $\lambda>0$.}
\end{equation}
The following holds.
\begin{prop}
    Let $B$ as in \eqref{eq:decB} with $\gamma \in (0,1)$. For each $\lambda>0$ we have that $\mathcal{R}(\lambda- A)$ is dense in $C_0(\R)$.
\end{prop}
\begin{proof}
We start by noticing that the space $C^1_0(\R)=\{\vp \in C^1(\R) \; | \; \lim_{|v|\rightarrow +\infty}|\vp(v)|=0\}$ is dense in $C_0(X)$.
Now let $\vp \in \D(A), g\in C^1_0(\R) $, the resolvent equation reads
\begin{equation}
    \lambda \vp+Kv_2 \partial_{v_1}\vp+\nu\vp-\Ll^*_+(\vp)=g, \quad \lambda>0.
\end{equation}
Integrating yields
  \begin{gather}
    \vp(v)=\left(\frac{1}{K|v_2|}\bigintsss_{-\infty}^{v_1} \exp\left(-\frac{1}{K|v_2|}\int_s^{v_1} (\lambda + \nu(|\zeta_z(v)|)) dz \right)\Ll^*_+(\vp)(\zeta_s(v))ds \right.   \notag \\
    + 
    \left. \frac{1}{K|v_2|}\bigintsss_{-\infty}^{v_1} \exp\left(-\frac{1}{K|v_2|}\int_s^{v_1} (\lambda + \nu(|\zeta_z(v)|))dz\right)g(\zeta_s(v))ds\right) \chi_{\{v_2 > 0\}} \notag \\
    + \left( \frac{(\Ll^*_+ \vp)(v_1,0,v_3)}{\lambda+\nu(|(v_1,0,v_3)|)} + \frac{g(v_1,0,v_3)}{\lambda + \nu(|(v_1, 0, v_3)|)}\right)\chi_{\{v_2=0\}} \notag \\
   + \left(\frac{1}{K|v_2|}\bigintsss_{v_1}^{+\infty} \exp\left(-\frac{1}{K|v_2|}\int_s^{v_1} (\lambda + \nu(|\zeta_z(v)|)) dz \right)\Ll^*_+(\vp)(\zeta_s(v))ds \right.  \notag \\
    \left. + \frac{1}{K|v_2|}\bigintsss_{v_1}^{+\infty} \exp\left(-\frac{1}{K|v_2|}\int_s^{v_1} (\lambda + \nu(|\zeta_z(v)|))dz\right)g(\zeta_s(v))ds\right) \chi_{\{v_2 <0\}} \label{eq:resolvent}
    \end{gather}
    with
    \begin{equation}
    \zeta_t(v)=(v_1+t,v_2,v_3).
    \end{equation}
    The function $\vp(v)$ defined by \eqref{eq:resolvent} is continuous since we have 
    \begin{align}
        \lim_{v_2 \rightarrow 0^+} \frac{1}{K|v_2|}\bigintsss_{-\infty}^{v_1} \exp\left(-\frac{1}{K|v_2|}\int_s^{v_1} (\lambda + \nu(|\zeta_z(v)|)) dz \right)f(\zeta_s(v))ds  & =   \frac{f(v_1,0,v_3)}{\lambda+\nu(|(v_1,0,v_3)|)} \\
        \lim_{v_2 \rightarrow 0^-} \frac{1}{K|v_2|}\bigintsss_{v_1}^{+\infty} \exp\left(-\frac{1}{K|v_2|}\int_s^{v_1} (\lambda + \nu(|\zeta_z(v)|)) dz \right)f(\zeta_s(v))ds &= \frac{f(v_1,0,v_3)}{\lambda+\nu(|(v_1,0,v_3)|)}
    \end{align}
    for every $f \in C_0(X)$. We claim that
    \begin{align}
        \Bigg | \frac{1}{Kv_2} \bigintsss_{-\infty}^{v_1} \exp\left(-\frac{1}{Kv_2}\int_s^{v_1} (\lambda + \nu(|\xi_z(v)|)) dz \right)\Ll^*_+(\vp)(\xi_s(v))ds \Bigg | & <  \|\vp\|_{\infty}; \notag \\
        \Bigg | \frac{1}{K|v_2|} \bigintsss_{v_1}^{+\infty} \exp\left(-\frac{1}{K|v_2|}\int_s^{v_1} (\lambda + \nu(|\xi_z(v)|)) dz \right)\Ll^*_+(\vp)(\xi_s(v))ds \Bigg | & <  \|\vp\|_{\infty}; \notag \\
         \Big |\frac{(\Ll^*_+ \vp)(v_1,0,v_3)}{\lambda+\nu(|(v_1,0,v_3)|)} \Big |  & < \|\vp\|_{\infty} 
    \end{align}
    for every $\vp \in C_0(\R).$ We prove the claim for $v_2>0$ since the case $v_2<0$ is analogous and since the case $v_2=0$ is trivial. From \eqref{eq:resolvent} we have that\\
    \begin{align}
        & \Bigg |\frac{1}{Kv_2} \bigintsss_{-\infty}^{v_1} \exp\left(-\frac{1}{Kv_2}\int_s^{v_1}(\lambda + \nu(|\xi_z(v)|)) dz \right)\Ll^*_+(\vp)(\xi_s(v))ds \Bigg |   \notag \\
       &\leq \|\vp\|_{\infty} \frac{1}{Kv_2}\bigintsss_{-\infty}^{v_1} \exp\left(-\int_s^{v_1} (\lambda + \nu(|\xi_z(v)|)) dz \right)\nu(|\xi_s(v)|)ds  \notag \\  & = \|\vp\|_{\infty} \left(\frac{1}{Kv_2}\bigintsss_{-\infty}^{v_1} \frac{d}{ds} \left(\exp\left(-\frac{1}{Kv_2}\int_s^{v_1} (\lambda + \nu(|\xi_z(v)|)) dz \right)\right) ds \right.  \notag \\
       &  \left . -\frac{\lambda}{Kv_2} \bigintsss_{-\infty}^{v_1} \exp\left(-\frac{1}{Kv_2}\int_s^{v_1} (\lambda + \nu(|\xi_z(v)|)) dz \right) \right) \notag \\
       & \leq \|\vp\|_{\infty}\left(1-\exp\left(-\frac{1}{Kv_2}\int_{-\infty}^{v_1} (\lambda + \nu(|\xi_s(v)|)) ds \right)\right).
    \end{align}
    Since we have $\lambda+\nu(|v|)>0$ for each $\lambda>0$ and from the fact that $\sup_{v \in \R}\int_{-\infty}^{+\infty}(\lambda+\nu(|\xi_s(v)|))ds=+\infty$ for each $\lambda>0$, it follows that  
    \begin{equation}
        1-\exp\left(-\frac{1}{kv_2}\int_{-\infty}^{v_1} (\lambda + \nu(|\xi_s(v)|)) ds \right) < 1, \quad \lambda>0.
    \end{equation}
    This proves the claim. Therefore we get that there exist a unique solution to \eqref{eq:resolvent} for each $\lambda>0$.
    \end{proof}
    
    From the previous proposition we obtain that the closure of the operator $(A,\D(A))$ is a Markov generator. We can then apply Hille-Yosida theorem \ref{thm:HY}) to get that $(\bar A,\D(\bar A))$ generates a Markov semigroup $S(t)$ and that the Cauchy problem \eqref{eq:adjCauchy2} is well-posed on $C_0(\R)$, see \cite{Pazy}.
    Following the same lines as in the proof of Theorem \ref{thm:wp2} in Section \ref{sec:3.1} and from our choice of $\D(A)$ we can then conclude the proof of Theorem \ref{thm:wpGamma>0}.

\section{Existence of a Stationary non-equilibrium solution   for the Boltzmann-Rayleigh equation with simple shear deformations.}

\subsection{The case of pseudo Maxwell molecules}\label{ssec:Maxwell} 

We recall that for pseudo Maxwell molecules the collision kernel $B$ takes the form
\begin{equation}\label{eq:assb} 
    B(n \cdot \omega,|v|)=b(n \cdot \omega) \quad \text{with $b \in L^{\infty}(S^2)$}
\end{equation}

The main result of this section is the following
\begin{thm}\label{thm:ssMM}
    Let $0 \leq K < l_0$ and let $f_0 \in \mathscr{M}_{s,+}(\R)$.  Suppose that the collision kernel $B$ is as in \eqref{eq:assb}. Let $f \in C([0,+\infty),\M)$ be the unique weak solution obtained in Theorem \ref{thm:wp2} then there exists a stationary solution of the Cauchy problem \eqref{eq:Cauchy2} i.e. a fixed point for the adjoint $S^*(t)$ of the semigroup introduced in \eqref{eq:semiS}. 
\end{thm}

To prove this result we will rely on the analysis of the moments $M_{jk}(t)=\int_{\R} v_jv_k f(t,dv)$,  $j,k=1,2,3$ of the solution to \eqref{eq:adjCauchy2}. Our goal is to construct a subspace $X \subseteq \M$ invariant under the action of the solution semigroup $S(t)$ obtained in \eqref{eq:semiS} which, under suitable conditions, will allow us to apply Schauder fixed point theorem. To this end will first derive a system of ODEs for $M_{jk}(t)$ and we will then show that the spectrum $\sigma(M)$ of $M=\left(M_{jk}(t)\right)_{j,k=1,2,3}$ is contained in the complex left half-plane, for $K \in [0,K_0]$. This implies the system of ODEs is stable and that the matrix norm of $M(t)$ remains bounded for all $t \geq 0$.

For this purpose we will use a modified version of   \eqref{eq:weakf}, which is essentially a weak formulation not integrated in time. In fact, let $\vp(v) \in C_c^{\infty}(\R)$. Multiplying \eqref{eq:adjCauchy2} by $\vp$ and integrating over $\R$ yields
\begin{equation}\label{eq:weakFormTimeDep}
    \partial_t \left(\int_{\R}f(t,dv)\vp(v)\right)+\int_{\R}Kv_2 \partial_{v_1} \vp(v)f(t,dv)=\int_{\R}\int_{\R}\int_{S^2}b (n \cdot \omega)\, f(t,dv)M_*(\vp(v')-\vp(v))dv_*d \omega.
\end{equation}

\subsubsection{Moments Equation for Maxwell Molecules}\label{sec:3}

We begin by proving a lemma that will be needed in the following.

\begin{lem}\label{lem:tensor}
    Let the collision kernel $B$ as in \ref{eq:assb}. Let $a,b \in \R$ and define the tensor product $a \otimes b$ as the bilinear transformation acting as
    \begin{equation}
        (a \otimes b)x=(a \cdot x) b \quad x \in \R.
    \end{equation}
    Then
    \begin{align}
        & \int_{S^2}  b(n \cdot \omega)(V \cdot \omega)^2(\omega \otimes \omega) d \omega  
        = \frac{\beta-\alpha}{2} |V|^2I+\frac{3\alpha-\beta}{2} V \otimes V; \label{eq:tensorQuadratic} \\
        &\int_{S^2}  b(n \cdot \omega)(V \cdot \omega)(\omega \otimes v+v \otimes \omega) d \omega  = \beta (V \otimes v+v \otimes V) \label{eq:tensorAnnoying} 
    \end{align}
    where $V=v-v_*$ and
    \begin{align}
        \alpha & = \int_{S^2} b (e_1 \cdot \omega)(e_1 \cdot \omega)^4 d \omega = 2 \pi \int_0^{\pi} b(\cos \theta) \cos^4 \theta \sin \theta d \theta = 2 \pi\int_{-1}^1 b(x)x^4 dx; \label{eq:alpha} \\
        \beta & = \int_{S^2} b (e_1 \cdot \omega)(e_1 \cdot \omega)^2 d \omega = 2 \pi \int_0^{\pi} b(\cos \theta) \cos^2 \theta \sin \theta d \theta = 2 \pi\int_{-1}^1 b(x)x^2 dx. \label{eq:beta}
    \end{align}
\end{lem}

\begin{proof}
    We first prove \eqref{eq:tensorAnnoying}, to begin with we recall that $n=\frac{V}{|V|}$. There exists a an orthogonal matrix with positive determinant such that $\frac{V}{|V|}=Re_1$, moreover setting $\omega=R \tilde \omega$ yields
    \begin{equation}
       A = \int_{S^2}  b(n \cdot \omega)(V \cdot \omega)(\omega \otimes v) d \omega = |V| \int_{S^2} b(e_1 \cdot \tilde \omega)(e_1 \cdot  \tilde \omega )((R \tilde \omega) \otimes v) d \tilde \omega.
    \end{equation}
    Notice that we have that 
    \begin{align}
        A Re_1 & = \left(|V|^2 \int_{S^2} b(e_1 \cdot  \tilde \omega) (e_1 \cdot  \tilde \omega )^2 d \tilde \omega\right) v = \beta |V|^2v; \\
        A Re_j & = \left(|V| \int_{S^2} b(e_1 \cdot  \tilde \omega) (e_1 \cdot  \tilde \omega)(e_j \cdot \tilde \omega ) d \tilde \omega\right) v = 0 \quad j=2,3
    \end{align}
    where 
    \begin{equation}
        \beta = \int_{S^2} b (e_1 \cdot \tilde \omega)(e_1 \cdot \tilde \omega)^2 d \tilde \omega = 2 \pi \int_0^{\pi} b(\cos \theta) \cos^2 \theta \sin \theta d \theta = 2 \pi\int_{-1}^1 b(x)x^2 dx.
    \end{equation}
    Hence for any vector $x \in \R $ we have that 
    \begin{equation}
        Ax= \sum_{i=1}^3 Re_i(Re_i \cdot Ax)= Re_1 (Re_1 \cdot Ax) = \beta (v \cdot x) V 
    \end{equation}
    which implies that 
    \begin{equation}
        A= \beta (V \otimes v).
    \end{equation}
    The same argument shows that
    \begin{equation}
        A^T =  |V| \int_{S^2} b(e_1 \cdot  \tilde \omega)(e_1 \cdot \tilde \omega)(v \otimes (R \tilde \omega)) d \tilde \omega= \beta (v \otimes V)
    \end{equation}
    from which \eqref{eq:tensorAnnoying} is proven. Now we prove \eqref{eq:tensorQuadratic}. Using the same rotation argument as before we get
    \begin{align}
       B & = \int_{S^2}  b(n \cdot \omega)(V \cdot \omega)^2(\omega \otimes \omega) d \omega = |V|^2 \int_{S^2} b( e_1 \cdot \tilde \omega)(e_1 \cdot \tilde \omega)^2((R\tilde \omega)\otimes (R\tilde \omega)) d \tilde\omega.
    \end{align}
    Therefore we have that
    \begin{align}
        Re_k B Re_j & = |V|^2 \int_{S^2} b( e_1 \cdot \tilde \omega)(e_1 \cdot \tilde \omega)^2((R\tilde \omega)(Re_k) (R\tilde \omega)(Re_j)) d\tilde\omega \notag \\
       & = |V|^2\int_{S^2} b( e_1 \cdot \tilde \omega)(e_1 \cdot \tilde \omega)^2(e_j \cdot \tilde \omega) (e_k \cdot \tilde \omega) d \tilde \omega \\
       & = \begin{cases}
           0 & j \neq k; \\
            |V|^2\int_{S^2} b( e_1 \cdot \tilde \omega)(e_1 \cdot \tilde \omega)^2(e_j \cdot \tilde \omega)^2 d \tilde  \omega & j=k.
       \end{cases}
    \end{align}
    Furthermore the integral for $j=1$ reads
    \begin{equation}
        |V|^2\int_{S^2} b( e_1 \cdot \tilde \omega)(e_1 \cdot \tilde \omega)^2(e_j \cdot \tilde \omega)^2 d \tilde  \omega  =  |V|^2\int_{S^2} b( e_1 \cdot \tilde \omega)(e_1 \cdot \tilde \omega)^4 d \tilde \omega
    \end{equation}
    from which it follows that $B$ is diagonal with $B_{22}=B_{33}$ by symmetry. Moreover 
    \begin{align}
        2 B_{22} & = |V|^{2} \int_{S^2} b(e_1 \cdot \tilde \omega)(e_1 \cdot \tilde \omega)^2((e_2 \cdot \tilde \omega)^2+(e_3 \cdot \tilde \omega)^2) d \tilde \omega = |V|^2 \int_{S^2} b(\tilde \omega_1) (\tilde \omega_1)^2((\tilde \omega_2)^2+(\tilde \omega_3)^2) d \tilde \omega \notag \\
        & = |V|^2 \int_{S^2} b(\tilde \omega_1) \tilde \omega_1^2(1-\tilde \omega_1^2) d \tilde \omega = |V^2|(\beta-\alpha) 
    \end{align}
    thus
    \begin{equation}
        B= \frac{\beta-\alpha}{2} |V|^2I+\frac{3\alpha-\beta}{2} |V|^2 \delta_{k1}\delta_{j1}= \frac{\beta-\alpha}{2} |V|^2I+\frac{3\alpha-\beta}{2} V \otimes V.
    \end{equation}

\end{proof}

First we notice that the moment matrix $M$ is symmetric and thanks to \eqref{eq:weakFormTimeDep} it is possible to derive an evolution equation for $M$. We have the following.

\begin{prop}\label{lem:derM}
   Let the collision kernel $B$ as in \eqref{eq:assb}. Let $f \in C([0,+\infty),\M)$ be such that
    \begin{equation}
        \int_{\R}|v|^2f(t,dv)< + \infty \quad \text{for all $t \geq 0$}
    \end{equation}       
     
     Then the moments $M_{jk}$ satisfy the following evolution equation
    \begin{align}
            \frac{d}{dt}M_{jk} & +K(\delta_{j1}M_{2k}+\delta_{k1}M_{2j}) = \left(\frac{1}{2}(3 \alpha-\beta)-2 \beta \right) M_{jk}  \notag \\
           & + \frac{1}{2}(\beta- \alpha) \sum_{l=1}^3 M_{ll}\delta_{jk}+ \frac{1}{2}(\beta- \alpha)\delta_{jk} + \frac{1}{6}(3\alpha-\beta)\delta_{jk} \label{eq:evolutionM}
    \end{align}
    where $\delta_{jk}$ is the Kronecker delta and $\alpha, \beta$ have been defined in \eqref{eq:alpha},\eqref{eq:beta}.
    
\end{prop}

\begin{remark}
   We will present in Appendix  \ref{appendixB} the proof of the Proposition above in the case in which the collision kernel $B$ satisfies \eqref{eq:assb} with $b=1$ because in this case the moments equations \eqref{eq:evolutionM} have a simpler form, due to the fact that the constants $\alpha, \, \beta$  are explicit. 
\end{remark}

\begin{proof}
    Choosing $\vp(v)=v_jv_k$ as test function in \eqref{eq:weakFormTimeDep} and setting $V=v-v_*$ gives
    \begin{gather}\label{eq:moments1}
        \frac{d}{dt} M_{jk}+K\int_{\R}f(t,dv) v_2\partial_{v_1}(v_jv_k)=\notag \\
        \int_{\R}\int_{\R}\int_{S^2}b (n \cdot \omega)\, f(t,dv)M_*\left[(v-(V \cdot \omega)\omega)_j(v-(V \cdot \omega)\omega)_k-v_jv_k\right]dv_* d\omega .
    \end{gather}
Computing the derivative on the left-hand side yields
\begin{gather}
    \frac{d}{dt}M_{jk}+K \int_{\R} f(t,dv)v_2v_j\delta_{j1}+K \int_{\R} f(t,dv)v_2v_k\delta_{k1}= \notag \\
     \frac{d}{dt}M_{jk}+K(\delta_{j1}M_{2k}+\delta_{k1}M_{2j}).
\end{gather}
Computing explicitly the right-hand is more involved. To to begin with we have that
\begin{align}
    &\int_{S^2}b(n \cdot \omega) \left[(v-(V \cdot \omega)\omega)_j(v-(V \cdot \omega)\omega)_k-v_jv_k\right] d \omega \notag \\ 
    & = -\int_{S^2}b(n\cdot \omega) \left[(V \cdot \omega)\omega_jv_k+(V \cdot \omega)\omega_kv_j\right]d \omega +\int_{S^2}b(n \cdot \omega)(V \cdot \omega)^2\omega_j\omega_k d \omega \notag \\
    & = -\int_{S^2} b(n \cdot \omega) (V \cdot \omega)\left[\omega \otimes v +v \otimes \omega \right] d\omega + \int_{S^2} b(n \cdot \omega) (V \cdot \omega)^2 \omega \otimes \omega d \omega.
\end{align}
which is a second order tensor in $\omega$. Lemma \ref{lem:tensor} then gives 
\begin{equation}
    \int_{S^2}b(n \cdot \omega) \left[(v-(V \cdot \omega)\omega)_j(v-(V \cdot \omega)\omega)_k-v_jv_k\right] d \omega \notag  = -\beta(v \otimes V + V \otimes v)+\frac{\beta-\alpha}{2}|V|^2+\frac{3\alpha-\beta}{2} V \otimes V.
\end{equation}

Therefore the linear term in $V$ yields
\begin{align}
    -\beta\int_{\R}f(t,dv)\int_{\R}M_*\left[v_j(v-v_*)_k+v_k(v-v_*)_j\right] dv_* = -2 \beta M_{jk}
\end{align}
while the quadratic term in $V$ gives
\begin{align}
    &\int_{\R}f(t,dv)\int_{\R} M_* \left[ \frac{\beta-\alpha}{2}(v-v_*)^2\delta_{jk}+\frac{3 \alpha-\beta}{2}(v-v_*)_j(v-v_*)_k\right]dv_* 
    \notag \\
    & \quad =\left(\frac{\beta-\alpha}{2} \sum_{i=1}^3 M_{ii}+\frac{\beta-\alpha}{2}\right) \delta_{jk}  + \frac{3 \alpha-\beta}{2} M_{jk}+ \frac{3 \alpha - \beta}{2}\int_{\R} M_*(v_*)_j(v_*)_k d v_* \notag \\
    &  \quad =
    \left(\frac{\beta-\alpha}{2} \sum_{i=1}^3 M_{ii}+\frac{\beta-\alpha}{2}\right)\delta_{jk} + \frac{3 \alpha-\beta}{2} M_{jk}+ \frac{3 \alpha-\beta}{6} \delta_{jk}.
\end{align}
All together \eqref{eq:moments1} then reads
\begin{align}
\label{eq:evMjk}
            &\frac{d}{dt}M_{jk}+K(\delta_{j1}M_{2k}+\delta_{k1}M_{2j})  \notag \\
           =  \left(\frac{1}{2}(3 \alpha-\beta)- 2 \beta \right) M_{jk} + & \frac{1}{2}(\beta- \alpha) \sum_{i=1}^3 M_{ii}\delta_{jk}+ \frac{1}{2}(\beta- \alpha)\delta_{jk} + \frac{1}{6}(3\alpha-\beta)\delta_{jk}.
\end{align}

\end{proof}

\begin{remark}\label{rem:3}
    The previous result is formal since the functions $v_jv_k$ are not in the space of test functions $C_c^{\infty}(\R)$. However, we remark that it can be made rigorous by considering suitable cut-offs of the functions $v_jv_k$. In fact, consider the sequence of functions
    \begin{equation}
        \vp_n(v) = 
        \begin{cases}
            e^{-\frac{1}{(n^2-|v|^2)n^5}} & \text{if $|v|\leq n$,} \\
            0 & \text{if $|v|>n$.}
        \end{cases}
    \end{equation}
    We have that $(\vp_n)_n \subseteq C_c^{\infty}(\R)$, $0 \leq \vp_n \leq 1$, $\vp_n \rightarrow 1$ monotonically in $n$. Moreover we have $|\partial_v \vp_n(v)| \leq \frac{C}{n^4}$. Hence choosing as test function $|v|^2\vp_n(v)$ in \eqref{eq:weakFormTimeDep} and passing to the limit as $n \rightarrow + \infty$ gives the above results. In the following this remark  has to be understood as the underlying procedure of the proofs.
\end{remark}

\begin{remark}
    Notice that comparing \eqref{eq:evMjk} with its analog in the nonlinear case the tensors appearing in the right-hand side are different. This is due to the fact that in the Boltzmann-Rayleigh case we do not have energy conservation which implies that the tensor for the nonlinear case is traceless while it is the case in for the nonlinear Boltzmann equation, see \cite{JNV1} Section 4.1 for the details in the nonlinear setting. More precisely in \eqref{eq:evMjk} there appear two constants $\alpha,\beta$ depending on $b$ which are not related while in the nonlinear only the constant $b=\int_{-1}^1b(x)x^2(1-x^2) dx$ appears.    \end{remark}

For the rest of this section we set $m= \mathrm{tr}M$. Consider \eqref{eq:evolutionM}, we have the followings
\begin{align}
    & \frac{1}{2}(3 \alpha-\beta)-2\beta  = \frac{3 \alpha-5 \beta}{2}; \\
   & \frac{1}{2}(\beta-\alpha)+\frac{1}{6}(3\alpha-\beta) = \frac{\beta}{3} 
\end{align}
where we notice that $\frac{3 \alpha-5 \beta}{2}<0$.
We now introduce
\begin{align}
    A(M) & = -(ML+L^{T}M) +\left(\frac{1}{2}(3 \alpha-\beta)- 2 \beta\right) M +      \frac{1}{2}(\beta- \alpha) m I \\
            &=-(ML+L^TM)+\frac{3 \alpha-5 \beta}{2}M+ \frac{\beta-\alpha}{2} m I;  \label{eq:defA}\\
     T(M) & =A(M)+\frac{1}{2}(\beta- \alpha)I + \frac{1}{6}(3\alpha-\beta)I= A(M)+ \frac{\beta}{3} I  \label{eq:defT}     
\end{align}
where $L$ is the shear matrix. Our aim is now to solve the equation 
\begin{equation}
    T(M)= \lambda M, \quad \lambda \in \mathbb{C}.
\end{equation}
which using \eqref{eq:defT} reads
\begin{equation}\label{eq:eigenvalueM}
   -K(\delta_{j1}M_{2k}+\delta_{k1}M_{2j}) +\frac{3 \alpha-5 \beta}{2} M_{jk}+\frac{\beta-\alpha}{2}m \delta_{jk}+\frac{\beta}{3} \delta_{jk}=\lambda M_{jk}.
\end{equation}

We have the following result. \
\begin{lem}\label{thm:spectruM}
    Consider the operator $A$ defined as in \eqref{eq:defA}. The characteristic polynomial of $A$ is
    \begin{equation}
        p_A(\lambda)=(\lambda+C_1)^3((\lambda+C_1)^3-3C_2(\lambda+C_1)^2-2C_2K^2)
    \end{equation}
    with 
    \begin{equation}
        C_1=-\frac{3 \alpha-5 \beta}{2}, \; C_2=\frac{\beta-\alpha}{2}.
    \end{equation}
    Moreover if $0 \leq K^2 < K_0$ then
    \begin{equation}
        \sigma(A) \subseteq \{z \in \mathbb{C} \; | \; \mathrm{Re}(z) <0\}.
    \end{equation}
\end{lem}

\begin{proof}
  Writing the equation $A(M)=\lambda M$ into into its components gives the following system of equations
    \begin{equation}\label{eq:eigSystem}
        \begin{cases}
            -2KM_{21}-C_1M_{11}+C_2m=\lambda M_{11}, & \\
            -KM_{22}-C_1M_{21}=\lambda M_{21}, & \\
            -KM_{23}-C_1M_{13}=\lambda M_{13}, & \\
            -C_1M_{22}+C_2m=\lambda M_{22}, & \\
            -C_1M_{23}=\lambda M_{23}, & \\
            -C_1M_{33}+C_2m=\lambda M_{33}.
        \end{cases}
    \end{equation}
    Since $M$ is symmetric it can be canonically viewed as a vector in $\mathbb{R}^6$ by setting
    \begin{equation}
        M= (M_{11},M_{12},M_{13},M_{22},M_{23},M_{33}).
    \end{equation}
    The matrix $A \in M_6(\mathbb{R})$ of the operator $M \mapsto A(M)$ then reads
    \begin{equation}
       A = \left( \begin{array}{cccccc}
            C_2-C_1 & -2K & 0 & C_2 & 0 & C_2  \\
           0  & -C_1 & 0 & -K & 0 & 0 \\ 
           0 & 0 &  -C_1 & 0 & -K & 0 \\
           C_2 & 0 & 0 & C_2-C_1 & 0 & C_2 \\
            0 & 0 & 0 & 0 & -C_1 & 0 \\
            C_2 & 0 & 0 & C_2 & 0 & C_2-C_1
        \end{array} \right)
    \end{equation}
    This allows to compute the spectrum with the standard determinant procedure. In fact, to compute the characteristic polynomial $p_A(\lambda)$ of $A$ we have that
     \begin{align}\label{eq:CharPolA}
      p_A(\lambda) & = \det ( A-\lambda I ) \notag \\
      & = \det  \left( \begin{array}{cccccc}
            C_2-C_1-\lambda & -2K & 0 & C_2 & 0 & C_2  \\
           0  & -C_1-\lambda & 0 & -K & 0 & 0 \\
           0 & 0 &  -C_1-\lambda & 0 & -K & 0 \\
           C_2 & 0 & 0 & C_2-C_1-\lambda & 0 & C_2 \\
            0 & 0 & 0 & 0 & -C_1-\lambda & 0 \\
            C_2 & 0 & 0 & C_2 & 0 & C_2-C_1 -\lambda
        \end{array} \right).
    \end{align}
    Expanding the determinant with respect to the third column gives
      \begin{gather}\nonumber
       p_A(\lambda)= \det(A-\lambda I)  \notag \\
      = -(\lambda+C_1)\det \left( \begin{array}{ccccc}
          C_2-C_1-\lambda  & -2K & C_2 & 0 & C_2  \\
          0  & -C_1-\lambda & -K & 0 & 0 \\
          C_2 & 0 & C_2-C_1-\lambda & 0 & C_2 \\
          0 & 0 & 0 & -C_1-\lambda & 0 \\
          C_2 & 0 & C_2 & 0 & C_2-C_1-\lambda
       \end{array}  \right)  \notag \\
      = (\lambda+C_1)^2 \left( \begin{array}{cccc}
          C_2-C_1-\lambda  & -2K & C_2 & C_2  \\
           0 & -C_1-\lambda & -K & 0 \\
           C_2 & 0 & C_2-C_1-\lambda & C_2 \\
           C_2 & 0 & C_2 & C_2-C_1-\lambda
       \end{array} \right) \notag \\
        = (C_1+\lambda)^2 \left[2K \det \left( \begin{array}{ccc}
           0 & -K & 0  \\
           C_2 & C_2-C_1-\lambda & C_2 \\
           C_2 & C_2 & C_2-C_1-\lambda
       \end{array}\right) \right. \notag \\
      \left. - (\lambda+C_1) \det \left(\begin{array}{ccc}
          C_2-C_1-\lambda  & C_2 & C_2  \\
          C_2  & C_2-C_1-\lambda & C_2 \\
          C_2 & C_2 & C_2-C_1-\lambda
       \end{array} \right) \right]  \notag \\
      = (\lambda+C_1)^2\left[2K^2(C_2(C_2-C_1-\lambda)-C_2^2)-(\lambda+C_1)(C_2-C_1-\lambda-C_2)^2(C_2-C_1-\lambda+2C_2) \right] \notag \\
      = (\lambda+C_1)^3((\lambda+C_1)^3-3C_2(\lambda+C_1)^2-2K^2C_2).
        \end{gather}
        Set $g(\lambda)=(\lambda+C_1)^3-3C_2(\lambda+C_1)^2-2K^2C_2$. The change of variables $y=\lambda+C_1$ gives $g(y)=y^3-3C_2y^2-2C_2K^2$. From the general formula for the discriminant of a cubic polynomial we get
        \begin{equation}
            \Delta_3=-216C_2^4K^2-108C_2^2K^4 < 0 \quad \text{for every $K > 0$}
        \end{equation}
        from which it follows that $g(y)$ always has only one real root $\bar y$. Furthermore from the fact that $g(0)=-2C_2K^2 <0$ and that the leading coefficient is positive, it follows that $\bar y >0$. Let now $z=a+ib$ be one the complex roots of $g$. Applying Viète's formula yields
        \begin{align}
            \bar y+2a & = 3C_2; \notag \\
            2\bar y a+(a^2+b^2) & = 0; \notag \\
            \bar y (a^2+b^2) & = 2C_2K^2 \notag 
        \end{align}
        from which it follows $-2 \bar y^2 a = 2C_2K^2$ and therefore $a<0$. Finally, going back to the $\lambda$ variable and plotting $\bar y$ we see, by continuity of the roots of $g(y)$ with respect to $K$, that $\bar y -C_1 <0 $ if and only if $K^2 < K_0$, which concludes the proof.
        
\end{proof}

\begin{remark}\label{rem:K0}
    Notice that the results of the previous Lemma remain holds true also in absence of the shear deformation, i.e. when if $K=0$. In fact, we have 
    $$p_A(\lambda)=(\lambda+C_1)^5(\lambda+C_1-3C_2) $$
    and the roots of $p_A$ are $\lambda_1=-C_1=\frac{3\alpha-5\beta}{2}, \; \lambda_2=3C_2-C_1= 4 \beta-3\alpha<0$ .
\end{remark}

Lemma \ref{thm:spectruM} and Remark \ref{rem:K0} allow us to prove the existence and uniqueness of a solution to the Cauchy problem related to \eqref{eq:evolutionM}. Moreover we can easily prove the stability of the solution, which will later lead to the control of the evolution of the moments for the unique solution of \eqref{eq:Cauchy2}.

\subsubsection{Existence of a stationary solution (Proof of Theorem \ref{thm:ssMM})}

 We start with next theorem. 

\begin{thm}\label{thm:solutionM}
    Let $0 \leq K^2 <  K_0$. There exists a unique solution $M(t)$ to   \eqref{eq:evolutionM} for every initial datum $f_0 \in \M$ such that
    \begin{equation}
        \int_{\R} f_0(dv)=1, \quad \int_{\R} |v|^2 f_0(dv) < +\infty.
    \end{equation}
    Moreover we have $\|M(t)\|_{\infty} \leq \bar C \|M_0\|_{\infty}+\tilde C < + \infty$ for every $t \geq 0$ with $\bar C, \tilde C >0$. 
\end{thm}

\begin{proof}
    A solution to the Cauchy problem 
    \begin{equation}\label{eq:ChauchyM}
        \begin{cases}
            \frac{d}{dt}M(t)=A(M)(t)+C_3 I & \\
            M(0)=M_0
        \end{cases}
    \end{equation}
    is obtained by adding a solution of the homogeneous equation and solution to the stationary non-homogeneous one. Namely
    \begin{equation}\label{eq:explicitM}
        M(t)=\tilde M e^{tA}+M_{\mathrm{st}} \quad \text{with $\tilde M \in M_3(\mathbb{R}), \  M_{\mathrm{st}}=-C_3 A^{-1}(I) $.}
    \end{equation}
    Notice that by Lemma \ref{thm:spectruM} and by Remark \ref{rem:K0} all the eigenvalues of $A$ have negative real part, hence $M_{\mathrm{st}}$ is well-defined. To determine $\tilde M$ we impose $M(0)=M_0$ which yields $\tilde M= M_0-M_{\mathrm{st}}$. Therefore we have
    \begin{equation}
        M(t)=M_0e^{tA}-M_{\mathrm{st}}(e^{tA}-I).
    \end{equation}
    Again by Lemma \ref{thm:spectruM}, by Remark \ref{rem:K0} and from the fact that $\int_{\R}f_0(dv) < + \infty$ it follows that $\|M(t)\|_{\infty} \leq \bar C \|M_0\|_{\infty}+C < + \infty$ which concludes the proof.
\end{proof}

\bigskip

\begin{remark}
    Notice that if $M_0$ is positive definite then $M(t)$ is positive definite as well. This can be readily seen from \eqref{eq:explicitM}, from the fact that under the assumptions of Theorem \ref{thm:spectruM} $\sigma(A) \subseteq \{z \in \mathbb{C} \; | \; \mathrm{Re}(z)<0\}$ and  that $\det A >0$ when $A$ is seen as a matrix in $\mathrm{M}_6(\mathbb{R})$.
\end{remark}

\begin{lem}[Povzner's Estimate]\label{lem:Povzner}
    Let $s>2$ and let $v^\prime\in\mathbb{R}^3$ be as in   \eqref{eq:collisonRule}. There exist a constant $C_s>0$ depending only on $s$ such that, for any $v\in\mathbb{R}^3$, the following inequality holds:
\begin{equation}\label{eq:Povzner}
        |v'|^s-|v|^s \leq -|v|^s+C_s(|v|^{s-1}|v_*|+|v_*|^{s-1}|v|)\ .
    \end{equation}
\end{lem}
 
\begin{proof}
  This Lemma is a variation of the classical Povzner inequality, see \cite{JNV1}. To prove it we distinguish several cases. Consider first the case $\frac{1}{2}|v| \leq |v_*| \leq 2 |v|$. Then applying the collision rule yields
        \begin{align}
            |v'|^s-|v|^s & \leq -|v|^s+(|v|+|((v-v_*)\cdot\omega)\omega|)^s \notag \\
           & \leq -|v|^s+\bar C_s(|v|^s+|v_*|^s) \leq -|v|^s+C_s(|v|^{s-1}|v_*|+|v_*|^{s-1}|v|).
        \end{align}
  Let now $|v| \leq \frac{1}{2}|v_*|$, the case $|v_*| \leq \frac{1}{2}|v|$ is similar. If $|v'| \leq \frac{1}{2}|v|$ then we get
        \begin{equation}
            |v'|^s-|v|^s \leq \bar C_s |v|^s-|v|^s \leq C_s(|v|^{s-1}|v_*|+|v_*|^{s-1}|v|)-|v|^s. 
        \end{equation}
  Now if $|v'| \geq \frac{|v|}{2}$ using the triangular inequality and the fact that $(1+x)^s \leq 1 + C_sx$ for $x \in [0,1]$, we have that
  \begin{align}
      |v'|^s  & \leq |v'-v|^s+C_s|v||v'-v|^{s-1}; \\
      |v_*|^s &\geq |v_*-v|^s-C_s|v|^{s-1}|v_*-v|.
  \end{align}
  Since $|v| \leq \frac{1}{2}|v_*|$ it follows that $|v'-v| \leq C|v_*|$ from which we have that
  \begin{align}
      |v'|^s-|v|^s & \leq |v'-v|^s+C_s|v||v'-v|^{s-1}-|v|^s\leq |v'-v|^s+C_s|v||v_*|^{s-1}-|v|^s \notag \\
       \leq |v-v_*|^s+C_s|v||v_*|^{s-1}-|v|^s & \leq  |v_*|^s+C_s|v||v_*|^{s-1}-|v|^s \leq C_s(|v|^{s-1}|v_*|+|v_*||v|^{s-1})-|v|^s.     
  \end{align}
\end{proof}

\begin{lem}\label{lem:weakStarCont}
    Let the collision kernel $b$ satisfy assumption \eqref{eq:assb}. Let $f_0 \in \M$ be such that
    \begin{equation}\label{eq:assg0}
        \int_{\R}f_0(dv)=1, \quad \int_{\R} |v|^sf_0(dv)< + \infty \quad \text{with $s>2$.}
    \end{equation}
    Let $f \in C^0([0,+\infty),\M)$ be the unique weak solution of Theorem \ref{thm:wp2} with initial datum $f_0$ and let $S(t)$ be the strongly continuous semigroup defined in \eqref{eq:semiS}. Then $S^*(t)$ is weak$\ast-$continuous with respect to $t$ and it is uniformly continuous on compact sets $[0,T] $ for $T>0$. In particular $S^*(t)$ is uniformly continuous on compact sets $[0,T]$ over the subspace $\mathscr{M}_{+,s}(\R)$. 
\end{lem}

\begin{proof}
    By the duality formula \eqref{eq:dualityFormula} we have that 
    \begin{equation}
        \int_{\R} S(t)\vp(v)f_0(dv)=\int_{\R}\vp(v)f(t,dv)=\int_{\R}\vp(v)S^*(t)f_0(dv) \quad \text{for all $\vp \in C_0(\R)$}. 
    \end{equation}
    Therefore $f(t,dv)=S^*(t)f_0(dv)$. Since the semigroup $S(t)$ is strongly continuous a well-known results in semigroup theory gives that the adjoint semigroup $S^*(t)$ is weak$*-$continuous, see \cite{Nagel}. Moreover it easily follows that $S^*(t)$ is uniformly continuous on compact sets of the form $[0,T]$ and in particular $S(t)$ is uniformly continuous on sets of the form $[0,T] \times \mathscr{M}_{+,s}(\R)$. This concludes the proof.
\end{proof}

\begin{thm}\label{thm:moments}
    Let the collision kernel $B$ as in \eqref{eq:assb}. Let $2<s<3, K<l_0$ and $f_0 \in \M$ such that 
    \begin{equation}
        \int_{\R}f_0(dv)=1, \; \int_{\R}|v|^sf_0(dv)<+\infty, \; \int_{\R}v_jv_k f_0(dv)=M_{jk}(0).
    \end{equation}
    Then 
    \begin{equation}\label{eq:momentConservation}
        \int_{\R}f(t,dv)=1, \; \int_{\R}v_jv_k f(t,dv)=M_{jk}(t) \quad \text{for all $t \geq 0$.}
    \end{equation}
    Furthermore there exist $l_0>0$ such that if $K<l_0$ and if $\int_{\R}|v|^sf_0(dv) \leq C_*$ for $C_*>0$ then
    \begin{equation}\label{eq:sMomentConservation}
        \int_{\R}|v|^sf(t,dv) \leq C_* \quad \text{for all $t \geq 0$.}
    \end{equation}
\end{thm}
    \begin{proof}
    To prove conservation of mass take as test function $\vp(v)=1$, then \eqref{eq:weakFormTimeDep} yields
    \begin{equation}
        \frac{d}{dt} \int_{\R}f(t,dv)+K\int_{\R}v_2\partial_{v_1}(1)f(t,dv)= \int_{\R}\Ll^*(1)f(t,dv).
    \end{equation}
   Since $\Ll^*(1)=0$ we have
    \begin{equation}
        \frac{d}{dt}\int_{\R}f(t,dv)=0
    \end{equation}
    and the conservation of mass follows. 
    
    Due to Lemma \ref{lem:derM} the moments $M_{jk}$ satisfy the second group of identities in \eqref{eq:momentConservation}, therefore it only remains to prove \eqref{eq:sMomentConservation}. We now take as test function $\vp(v)=|v|^s$. Set $M_s(t)=\int_{\R} |v|^sf(t,dv)$, then we have
    \begin{gather}
        \frac{d}{dt} M_s+sK\int_{\R}v_2\partial_{v_1}(|v|^s)f(t,dv)=\int_{\R}f(t,dv)\int_{\R}\int_{S^2}b(n \cdot \omega) M_*(|v'|^s-|v|^s)dv_*d\omega.
    \end{gather}
    It is possible to estimate $K\int_{\R}v_2\partial_{v_1}f(t,dv)$ by $KM_s$ and applying Povzner's inequality as in Lemma \ref{lem:Povzner} gives
    \begin{align}
        \frac{d}{dt} M_s & \leq K M_s+\int_{\R}f(t,dv)\int_{\R}\int_{S^2}b(n \cdot \omega)M_*(-|v|^s+C_s(|v|^{s-1}|v_*|+|v_*|^{s-1}|v|))dv_*d\omega  \notag \\
       & \leq  K M_s- \|b\|_{L^1(S^2)} M_s + C_sM_{s-1}+ \tilde C_s M_1 \label{eq:MsEstimate}
    \end{align}
    Since \eqref{eq:momentConservation} implies $\int_{\R} |v|^2f(t,dv) \leq \tilde C$ for all $t\geq 0$ and from the fact that $s<3$ it follows that $C_sM_{s-1}+4 C_s M_1 \leq \tilde C_1$. Therefore we have
    \begin{equation}
        \frac{d}{dt}M_s \leq \tilde C_1 +(K-\|b\|_{L^1(S^2)})M_s.
    \end{equation}
    If we choose $K< l_0=\|b\|_{L^1(S^2)}$ then Gronwall yields
    \begin{equation}
        M_s(t) \leq M_s(0) \leq C_*
    \end{equation}
   which concludes the proof.
\end{proof}

With Proposition \ref{thm:moments} is now rather easy to prove the existence of the desired stationary
solution, as stated in Theorem \ref{thm:ssMM}, using
Schauder fixed point Theorem. 

\begin{proof}[Proof of Theorem  \ref{thm:ssMM}]
    Define
    \begin{equation}
        \mathscr{U}=\left \{f \in \M \; \Big | \; \int_{\R} f(dv)=1, \; \int_{\R} v_k v_k f(dv)=M_{jk}, \; \int_{\R}|v|^s f(dv) \leq C_* \right \}.
    \end{equation}
    The set $\mathscr{U}$ is convex and closed in the weak$*$ topology of $\M$. Moreover $\mathscr{U}$ is also weak$*$ compact. To see this notice that $\mathscr{U}$ is contained in the unit ball of $\M$ and the space $C_0(\R)$ is separable. Since $\mathscr{U}$ is weak$*$-closed by the Banach-Alaoglu Theorem it follows that $\mathscr{U}$ is weak$*$ compact. From Proposition \ref{thm:moments} we have that for any $h \geq 0$ $S^*(h)\mathscr{U} \subseteq \mathscr{U}$, hence the operator $S^*(h)$ is weak$*$ compact. Therefore we can apply Schauder fixed point Theorem to prove the existence of $f_*^{(h)}$ such that $S^*(h)f_*^{(h)}= f_*^{(h)}$ and by the semigroup property $S^*(mh)f_*^{(h)}= f_*^{(h)}$ for every $m \in \mathbb{N}$. Take a sub-sequence $\{h_k\}_k$ such that $h_k \rightarrow 0$ and its corresponding sequence of fixed points $\left \{f_*^{(h_k)}\right \}_k$. This sequence is compact in $\mathscr{U}$, since $\mathscr{U}$ itself is compact, and taking a sub-sequence if needed we have $f_*^{(h_k)}\rightarrow f_*$ for some $f_* \in \mathscr{U}$ as $k \rightarrow + \infty$. For any $t >0$ there exists a sequence of integers such that $m_kh_k \rightarrow t$. On one hand this yields
    $S^*(m_kh_k)f_*^{(h_k)}=S^*(h_k)f_*^{(h_k)}=f_*^{(h_k)} \rightarrow f_*$ while on the other hand
    \begin{equation}
        S^*(m_kh_k)f_*^{(h_k)}=(S^*(m_kh_k)f_*^{(h_k)}-S^*(t))f_*^{(h_k)}+S^*(t)f_*^{(h_k)}.
    \end{equation}
    By the weak$*$ continuity of Lemma \ref{lem:weakStarCont} it follow that the right-hand side converges to $S^*(t)f_*$ which gives
    \begin{equation}
        S^*(t)f_*=S^*(m_kh_k)f_*^{(h_k)}=f_*^{(h_k)} \rightarrow f_*
    \end{equation}
    for every $t \geq 0$. Therefore $f_*$ is a fixed point for $S^*(t)$ and the proof is concluded.  
\end{proof}

\subsubsection{Non existence of a stationary solution for large values of the shear parameter $K$} 

In this subsection we study the behavior of $M(t)$ when the shear parameter $K$ is larger that $K_0$. It turns out that since in this regime $A$ has a positive eigenvalue then $M(t)$ increases at most exponentially for $t \rightarrow + \infty$ unlike what happens in Section \ref{sec:3}.  Moreover, there is no stationary solution to \eqref{eq:Cauchy2}.  The main result of this subsection is the following

\begin{prop}\label{prop:noSteady}
    Consider the operator $A$ defined in \eqref{eq:defA}. Let the collision kernel $B$ be as in \eqref{eq:AssB1}  with $b$ satisfying \eqref{eq:assb} and let $K^2 > K_0$. Then for any $f_0 \in \M$ or, equivalently, any $M_0 \in \mathrm{M}_3(\mathbb{R})$ we have that $M(t) \sim e^{\mu t}M_0$  where $\mu$ is the unique eigenvalue with positive real part of $A$.  
\end{prop}

We first need the following result.

\begin{lem}
   Consider the operator $A$ defined in \eqref{eq:defA}. Let the collision kernel $B$ be as in \eqref{eq:assb} and let $K^2 > K_0$. Let $M^{\mu}$ be the unique eigenvector of $A$ associated to the unique eigenvalue $\mu$ with positive real part. Then 
    \begin{equation}
       M^{\mu} =  \left(\begin{array}{ccc}
           M^{\mu}_{11}  & M^{\mu}_{12} & 0 \\
            M^{\mu}_{12} & M^{\mu}_{22} & 0 \\
            0 & 0 & M^{\mu}_{33}
        \end{array}\right), \quad M^{\mu}_{11},M^{\mu}_{22},M^{\mu}_{33}>0, \; M^{\mu}_{11}M^{\mu}_{22}-(M^{\mu}_{12})^2 > 0
    \end{equation}
    and there exist $g \in \M$ such that $M^{\mu}_{ij}=\int_{\R}v_jv_k g(dv)$.
\end{lem}

\begin{proof}
    To prove the result  first we explicitly compute $M^{\mu}$ by means of the standard theorems on the continuity on $K$ of simple eigenvalues and eigenvectors for $K \rightarrow + \infty$ and then we prove the existence of the desired measure $g$. We recall that the characteristic polynomial of $A$ is 
    $$p_A(\lambda) = (\lambda+C_1)^3 ((\lambda+C_1)^3-3C_2(\lambda+C_1)^2-2K^2)=(\lambda+C_1)^3g(\lambda)$$ where $\mu$ is a simple root of $g$. Set $y=\lambda+C_1$ then $g$ reads $g(y)=y^3-3C_2y^2-2C_2K^2$. Since $\mu$ is a simple root of $g$, from the implicit function theorem it follows that $\mu$ depends continuously on $K$, $C_2$ being fixed. Moreover we notice that since $\mu>0$, from \eqref{eq:CharPolA} we must have $M_{13}=M_{23}=0$. We expect $\mu \sim (2C_2)^{\frac{1}{3}}K^{\frac{2}{3}},$ and $M_{22},M_{33} \sim K^{-\frac{2}{3}},M_{12} \sim K^{-\frac{1}{3}}$  as $K \rightarrow + \infty$. Hence we perform the change of variables 
    \begin{equation}
        \left(\begin{array}{c}
            M_{11}  \\
            M_{12}  \\
            M_{22}  \\
            M_{33}
        \end{array}\right) = \left(\begin{array}{c}
            m_{11}  \\
            K^{-\frac{1}{3}} m_{12}  \\
            K^{-\frac{2}{3}} m_{22}  \\
            K^{-\frac{2}{3}} m_{33}
        \end{array}\right).
    \end{equation}
    Denote with $A_0$ the matrix obtained from $A$ removing rows and columns relative to $M_{13},M_{23}$, then we look for solutions to the equation
    \begin{equation}
        \left( \begin{array}{cccc}
           C_2-C_1  & -2K & C_2 & C_2 \\
            0 & -C_1 & -K & 0 \\
            C_2 & 0 & C_2-C_1 & C_2 \\
            C_2 & 0 & C_2 & C_2-C_1
        \end{array}\right) \left(\begin{array}{c}
            m_{11}  \\
            K^{-\frac{1}{3}} m_{12}  \\
            K^{-\frac{2}{3}} m_{22}  \\
            K^{-\frac{2}{3}} m_{33}
        \end{array}\right) = K^{\frac{2}{3}} \theta \left(\begin{array}{c}
            m_{11}  \\
            K^{-\frac{1}{3}} m_{12}  \\
            K^{-\frac{2}{3}} m_{22}  \\
            K^{-\frac{2}{3}} m_{33}
        \end{array}\right)
    \end{equation}
    where we have set $\lambda=K^{\frac{2}{3}} \theta$. The system of equations then reads
    \begin{equation}\label{eq:reducedRescaledA}
        \begin{cases}
            (C_2-C_1)m_{11}-2K^{\frac{2}{3}}m_{12}+C_2K^{-\frac{2}{3}}m_{22}+C_2K^{-\frac{2}{3}}m_{33}=K^{\frac{2}{3}}\theta m_{11} & \\
            -C_1K^{-\frac{1}{3}}m_{12}-K^{\frac{1}{3}}m_{22}=K^{\frac{1}{3}}\theta m_{12} & \\
            C_2m_{11}+(C_2-C_1)K^{-\frac{2}{3}}m_{22}+C_2K^{-\frac{2}{3}}m_{33}=K^{\frac{2}{3}} \theta m_{22} & \\
            C_2m_{11}+C_2K^{-\frac{2}{3}}m_{22} + (C_2-C_1)K^{-\frac{2}{3}}m_{33}=K^{\frac{2}{3}} \theta m_{33}
        \end{cases}
    \end{equation}
    Taking the limit for $K \rightarrow + \infty$ and using the fact that for simple eigenvalues the eigenvectors vary continuously yields
    \begin{equation}
        \begin{cases}
            -2m_{12} = \theta m_{11} & \\
            -m_{22} = \theta m_{12} & \\
            C_2 m_{11} = \theta m_{22} & \\
            C_2 m_{11} = \theta m_{33}
        \end{cases}
    \end{equation}
    and $\theta^3=2C_2$, where we only consider the real value of $\theta$. The matrix $B_0$ associated to this linear system reads
    \begin{equation}\label{eq:B_0}
        B_0=\left(\begin{array}{cccc}
            0 & -2 & 0 & 0  \\
            0 & 0 & -1 & 0 \\
            C_2 & 0 & 0 & 0  \\
            C_2 & 0 & 0 & 0  
        \end{array}\right)
    \end{equation}
    Choosing $M_{11}=m_{11}$ as parameter we have that
    \begin{equation}
        M^{\mu}= \left(\begin{array}{c}
             M^{\mu}_{11}  \\
             M^{\mu}_{12} \\
             M^{\mu}_{22}  \\
             M^{\mu}_{33}
        \end{array}\right) = \left(\begin{array}{c}
             m_{11}  \\
            -K^{-\frac{1}{3}}\frac{\theta}{2}m_{11} \\
            K^{-\frac{2}{3}}\frac{\theta^2}{2} m_{11} \\
            K^{-\frac{2}{3}}\frac{\theta^2}{2} m_{11}
        \end{array}\right)
    \end{equation}
    where $M^{\mu}$ is the eigenvector associated to $\mu$. Choosing $m_{11}=1$ we get that $M^{\mu}_{11},M^{\mu}_{22},M^{\mu}_{33}>0$ and we get that 
    \begin{equation}
        M^{\mu}_{11}M^{\mu}_{22}-(M^{\mu}_{12})^2= \frac{\theta^2}{4}>0
    \end{equation}
    which concludes the first part of the proof.
    
    Now we prove that the elements of $M^{\mu}$ can be obtained from a measure $g \in \M$. In fact, take $F_0(v)$ be a smooth and positive function with $\int_{\R} F_0(|v|^2)|v|^2 dv = 3 $ and define
    \begin{equation}
        g(dv)=F_0(A_1v_1^2+A_2v_2^2+A_3v_3^3)dv+\beta \delta(v-\bar v)
    \end{equation}
    with $A_1,A_2,A_3>0$ and $\bar v = (1,1,0)$. Then we have
    \begin{equation}
    \begin{cases}
        M^{\mu}_{11}  = \int_{\R}F_0(A_1v_1^2+A_2v_2^2+A_3v_3^3)v_1^2 dv+ \beta; \\
        M^{\mu}_{22}  = \int_{\R}F_0(A_1v_1^2+A_2v_2^2+A_3v_3^3)v_2^2 dv+ \beta; \\
        M^{\mu}_{33}  = \int_{\R}F_0(A_1v_1^2+A_2v_2^2+A_3v_3^3)v_3^2 dv; \\
        M^{\mu}_{12}  = \beta.
    \end{cases}
    \end{equation}
    Performing the change of variables $\sqrt{A_i}v_i=y_i$ in the previous integrals yields
    \begin{equation}
        \begin{cases}
            M^{\mu}_{11}= \frac{1}{\sqrt{A_1^3 A_2 A_3}} + \beta; \\
            M^{\mu}_{11}= \frac{1}{\sqrt{A_1 A_2^3 A_3}} + \beta; \\
            M^{\mu}_{33}= \frac{1}{\sqrt{A_1 A_2 A_3^3}}; \\
            M^{\mu}_{12}= \beta.
        \end{cases}
    \end{equation}
    where we have used the fact that by symmetry we have
    \begin{equation}
        \gamma_i= \int_{\R}F_0(|v|^2)v_i^2 dv=\frac{1}{3} \int_{\R}F_0(|v|^2)|v|^2dv = 1.
    \end{equation}
    Solving the system by iterated substitution yields
    \begin{align}
        \beta & = M^{\mu}_{12}; \\
        A_3 & = \left(\frac{(M_{22}^{\mu}-\beta)(M_{11}^{\mu}-\beta)}{(M_{33}^{\mu})^4}\right)^{\frac{1}{5}}; \\
        A_2 & = \frac{A_3 M_{33}^{\mu}}{M_{22}^{\mu}-\beta} = \frac{(M_{33}^{\mu})^{\frac{1}{5}}(M_{11}^{\mu}-\beta)^{\frac{1}{5}}}{(M_{22}^{\mu}-\beta)^{\frac{4}{5}}};\\
        A_1 & = \frac{1}{A_2A_3^3(M_{33}^{\mu})^2} = \frac{(M_{11}^{\mu}-\beta)^{\frac{4}{5}}}{(M_{22}^{\mu}-\beta)^{\frac{1}{5}}(M_{33}^{\mu})^{\frac{1}{5}}}.
    \end{align}
    We notice that since $M_{ii}^{\mu}>0$, $\beta=M_{12}^{\mu}<0$ all the quantities are well-defined.
\end{proof}

\begin{proof}[Proof of Proposition \ref{prop:noSteady}]
    If the shear parameter $K$ is such that $K^2 > K_0$, Lemma \ref{thm:spectruM} yields the existence of a unique $\mu \in \sigma(A)$ with positive real part. Moreover the eigenspace $V_{\mu}$ associated to $\mu$ has dimension $1$. For any $M \in \mathrm{M}_3(\mathbb{R})$ we have the following decomposition
    \begin{equation}\label{eq:decM}
        M=\mathrm{P}_{V_{\mu}} M + \mathrm{P}_{Y} M 
    \end{equation}
    where $Y$ is the space spanned by all the others eigenvectors of $A$. We note that all the eigenvectors of $A$ form a basis for the space of symmetric $3 times 3$ matrices. Now we consider $M_0 \in \mathrm{M}_3(\mathbb{R})$,  using  \eqref{eq:decM}, we can then write $M(t)$ as
\begin{equation}\label{eq:exponentialGrowth}
        M(t) = e^{tA} M_0  = e^{tA}(\mathrm{P}_{V_{\mu}} M_0 + \mathrm{P}_{Y} M_0) = e^{\mu t} + O\left(e^{-\beta t}\right)  
    \end{equation}
    where $\beta >0$. Hence, from \eqref{eq:exponentialGrowth},  it follows that $M(t) \sim e^{\mu t}$ as $t \rightarrow + \infty$. Moreover, by the continuity of $e^{tA}$ we have that $M(t) \sim e^{\mu t}$ for any $M \in \mathrm{M}_3(\mathbb{R})$ such that $\|M-M_0\|<\ve$. This shows that the set of initial data such that \eqref{eq:exponentialGrowth} holds does not have zero measure.
\end{proof}

For completeness, we add a slight variation of Theorem \ref{thm:moments} in which we show that if $K>l_0$ the moments $M_s$ with $s>2$ increase at most exponentially. This allows us to justify the approximation given in \eqref{eq:approxVlarge} and the related conjecture stated in the introduction. 

\begin{prop}\label{prop:expMoments}
       Suppose that the collision kernel $B$ is as in \eqref{eq:AssB1}  with $b$ satisfying \eqref{eq:assb}.  Suppose that the shear parameter $K>l_0$ and let $2<s<3,$ and $f_0 \in \M$ such that 
    \begin{equation}
        \int_{\R}f_0(dv)=1, \; \int_{\R}|v|^sf_0(dv)<+\infty, \; \int_{\R}v_jv_K f_0(dv)=M_{jk}(0).
    \end{equation}
    Then 
    \begin{equation}
        \int_{\R}f(t,dv)=1, \; \int_{\R}v_jv_k f(t,dv)=M_{jk}(t) \quad \text{for all $t \geq 0$.}
    \end{equation}
    Moreover, let $s>2$, we have that
    \begin{equation}
        M_s(t) \leq C e^{t}, \quad C>0, \, t \geq 0.
    \end{equation}
\end{prop}

\begin{proof}
    The proof follows the same lines as the one of Theorem \ref{thm:moments}. Thus we have only to prove that $M_s(t) \leq C e^t$. From \eqref{eq:MsEstimate} we have that
    \begin{equation}
        \frac{d}{dt} M_s \leq K M_s-\|b\|_{L^1(S^2)} M_s+C_sM_{s-1}+4C_sM_1,
    \end{equation}
     Equation \eqref{eq:momentConservation} gives that $\int_{\R} |v|^2f(t,dv) \leq \tilde C$ for all $t\geq 0$ and being $s<3$ it follows that $C_sM_{s-1}+4 C_s M_1 \leq \tilde C_1$. This yields
    \begin{equation}
        \frac{d}{dt}M_s \leq \tilde C_1 +(K-\|b\|_{L^1(S^2)})M_s
    \end{equation}
    and by Gronwall lemma we obtain
    \begin{equation}
        M_s(t) \leq C e^t, \quad t \geq 0.
    \end{equation}
\end{proof}

\begin{remark}
    In the case of $K>l_0$ there are possible choices of the initial datum such that no stationary solutions exists, see Proposition \eqref{prop:noSteady}. In this case we conjecture the existence of a self-similar solution to the asymptotic approximation of \eqref{eq:adjCauchy2}, namely
    \begin{equation}
        \begin{cases}
            \partial_t f- Kv_2\partial_{v_1}f=\int_{S^2}b(n \cdot \omega)(f(v-(v \cdot \omega)\omega))-f(v))d \omega & \\
            f(0,v)=f_0(v)
        \end{cases}
    \end{equation}
    which is valid for very large velocities $|v|$. Such a self-similar solution would describe the behaviour of $f$, solution to \eqref{eq:adjCauchy2}, as $t\to \infty$ since the velocity $v$ would become very large. More precisely, the self-similar solution would have the form 
\begin{equation}\label{eq:self-similar}
f(t,v)= e^{-3\frac {\mu}2 t} F\left(\frac{v}{e^{\frac {\mu}2 t}}\right)
    \end{equation}
where $\mu$ is the eigenvalue introduced in Proposition \ref{prop:noSteady}. 
\end{remark}

\subsection{The case of collision kernels with homogeneity $\gamma> 0$}
 
Here we study the existence of a stationary solution to \eqref{eq:Cauchy2} for $\gamma \in (0,1)$. It turns out that in this case the second order moment $M_2(t)$ is globally bounded in time. This allows us to use the same techniques used in Subsection \ref{sec:3.1} to prove the following theorem which is the main result of this Section.

\begin{thm}\label{thm:stationary>0}
   Let the collision kernel $B$ as in \eqref{eq:decB}  with $\gamma \in (0,1)$. Let $f_0 \in \mathscr{M}_{2,+}(\R)$. Let $f \in C([0,+\infty),\M)$ be the unique weak solution obtained in Theorem \ref{thm:wpGamma>0}. Then there exists a stationary solution of the Cauchy problem \eqref{eq:Cauchy2} i.e. a fixed point for the adjoint $S^*(t)$ of the semigroup introduced in \eqref{eq:semiS}.
\end{thm}

\subsubsection{Existence of a stationary solution (Proof of Theorem \ref{thm:stationary>0})} \label{ssec:statgamma>0}

We prove here Theorem \ref{thm:stationary>0}. 

We will prove that the second order moment $M_2(t)$ is bounded globally in time. In order to do this we will need the following lemma.
\begin{lem}
   Let $B$ as in \eqref{eq:decB}, then we have 
\begin{equation}\label{eq:Povzner2}
     \int_{S^2}b(n \cdot \omega)(|v'|^2-|v|^2) d\omega \leq  -2\pi  C |v-v_*|^{2}+4 \pi C |v_*||v-v_*| \quad C>0.
    \end{equation}
\end{lem}

\begin{proof}
Using the collision rule \eqref{eq:collisonRule} we can estimate the integrand in the left hand side of \eqref{eq:Povzner2} as 
\begin{equation}
        |v'|^2-|v|^2=|(v-v_*) \cdot \omega|^2-2(v \cdot \omega)((v-v_*) \cdot \omega)= -|(v-v_*) \cdot \omega|^2-2(v_* \cdot \omega)((v-v_*) \cdot \omega).
\end{equation}
    Recalling that $(v-v_*) \cdot \omega = \cos\theta |v-v_*|$ and passing into spherical coordinates $(\theta,\varphi)$ with the North pole aligned on $e_3=(0,0,1)$ 
     with $\theta=0$, yields 
\begin{equation}
        \int_0^{\pi} \int_0^{2 \pi} b(\cos\theta)\sin \theta(-|v-v_*|^2\cos^2 \theta) d\vp d \theta =-2\pi \left(\int_{-1}^1 b(x) x^2 dx \right) |v-v_*| \, .
    \end{equation}
    The term $-2(v_* \cdot \omega)((v-v_*) \cdot \omega)$ can be estimated instead as $ 2|v_*||v-v_*|$. All together we have
    \begin{equation}
        \int_{S^2}(|v'|^2-|v|^2)d \omega \leq  -2\pi C |v-v_*|^{2}+4 \pi C |v_*||v-v_*|.
    \end{equation}
\end{proof}

\begin{lem}\label{lem:M2}
   Let the collision kernel $B$ be as in \eqref{eq:decB} with $\gamma \in (0,1)$. We have that $M_2(t) \in L^{\infty}([0,+\infty))$ for all $t \geq 0$ with $C>0$.
\end{lem}

\begin{proof}
Choosing as test function $\vp(v)=|v|^2$ in the weak formulation \eqref{eq:weakFormTimeDep} 
we have
\begin{equation*}
    \frac{d}{dt} M_{2}+K\int_{\R}f(t,dv) v_2\partial_{v_1}(\vert v\vert^2)= \int_{\R}\int_{\R}\int_{S^2}f(t,dv)M_*\left( \vert v^\prime\vert^2- \vert v \vert^2\right)dv_* d\omega .
\end{equation*}
Applying the estimate \eqref{eq:Povzner2} we obtain 
\begin{align}
    \frac{d}{dt}  M_2 & \leq 2K \int_{\R} v_2v_1f(t,dv)+\int_{\R}f(t,dv)\int_{\R}|v-v_*|^{\gamma}M_*( -2\pi C |v-v_*|^{2}+4 \pi C |v_*||v-v_*|)d v_* \notag \\ & \leq 
    2KM_2-2\pi C \int_{\R} f(t,dv)\int_{\R}|v-v_*|^{\gamma+2}M_*dv_* + 4\pi C\int_{\R} f(t,dv) \int_{\R}|v-v_*|^{\gamma+1}|v_*|M_* dv_*
\label{eq:estM2_hp}
\end{align}
We notice that 
\begin{equation}
    \int_{\R}|v-v_*|^{\gamma+2}M_* dv_* \geq C (1+|v|)^{2 + \gamma} \geq  C_1 |v|^2-C_2. 
\end{equation}
where $C_1,C_2>0$ and $C_1$ can be chosen arbitrarily large, while $C_2$ depends on $C_1$. Hence we have 
\begin{equation} \label{eq:estM2_hp1} \int_{\R} f(t,dv)  \int_{\R}|v-v_*|^{\gamma+2}M_* dv_* \geq C \int_{\R} f(t,dv)(1+|v|)^{2 + \gamma} \geq  C_1 \int_{\R} f(t,dv) |v|^2-C_2 
\end{equation}
where in the last inequality we used that $\int_{\R} f(t,dv)=1$. 
Considering now the third term in \eqref{eq:estM2_hp} by interpolation we obtain
\begin{equation} \label{eq:estM2_hp2}
   \int_{\R} f(t,dv) \int_{\R}|v-v_*|^{\gamma+1}|v_*|M_* dv_* \leq C M_{\gamma+1} \leq C M_2.
\end{equation}
Combining \eqref{eq:estM2_hp} with \eqref{eq:estM2_hp1} and \eqref{eq:estM2_hp2} we then arrive at
\begin{equation}
 \frac{d}{dt} M_2 \leq (2K+C-C_1) M_2+C_2   
\end{equation}
and since $C_1$ can be chosen arbitrarily large, more precisely $C_1>2K+C$ we can prove that $M_2$ is globally bounded by a simple ODE argument and the proof follows. 
\end{proof}

From the Theorem \ref{thm:wpGamma>0} it is easily follows the analogous of Lemma \ref{lem:weakStarCont} in the case $\gamma>0$, namely we have the following. 

\begin{lem}\label{lem:weakStarCont>0}
    Let the collision kernel $B$ satisfy assumption \eqref{eq:decB} with $\gamma \in (0,1)$. Let $f_0 \in \M$ be such that
    \begin{equation}\label{eq:assg0gamma>0}
        \int_{\R}f_0(dv)=1, \quad \int_{\R} |v|^2 f_0(dv)< + \infty 
    \end{equation}
    Let $f \in C^0([0,+\infty),\M)$ be the unique weak solution of Theorem \ref{thm:wpGamma>0} with initial datum $f_0$ and let $S(t)$ be the strongly continuous semigroup solving \eqref{eq:adjCauchy2}. Then $S^*(t)$ is weak$\ast-$continuous with respect to $t$ and it is uniformly continuous on compact sets $[0,T] $ for $T>0$. In particular $S^*(t)$ is uniformly continuous on compact sets $[0,T]$ over the subspace $\mathscr{M}_{2,+}(\R)$. 
\end{lem}

Define the set
\begin{equation}
    \mathscr{Y}=\left\{f \in \M \; \Big | \; \int_{\R}f(dv)=1, \int_{\R} |v|^2f(dv) \leq C \right\}
\end{equation}
for $C>0$. Relying on Lemma \ref{lem:M2} and Lemma \ref{lem:weakStarCont>0} we can then prove Theorem \ref{thm:stationary>0}, whose proof follows the same line of the one of Theorem \ref{thm:ssMM}.

\newpage

\appendix
\section{A remark on the consistency of homoenergetic  solution in the linear setting} \label{appA}

\bigskip 

We provide here more details on the passage from linear Boltzmann equation \eqref{A0_1} 
to the linear Boltzmann equation for homoenergetic flows \eqref{eq:homBol} by means of the homoenergetic ansatz \eqref{B1_0}-\eqref{B1_1}, validating its physical consistency. More precisely we will see that the ansatz \eqref{B1_0}-\eqref{B1_1} is the one that  provides the correct physical model, as counterpart in the linear setting of the nonlinear model. As discussed in the introduction, homoenergetic solutions have been originally introduced in \cite{Galkin1,Galkin2,Galkin3,T,TM}  and rigorously studied in \cite{CercArchive, JNV1} in the nonlinear setting as a particular class of solutions to 
\begin{gather}
\partial_{t}g+w \cdot \partial _{x}g =Q(g,g)\left( w\right), \quad
g=g\left( x,w,t\right)   \nonumber \\
Q(g,g)\left( w\right)  =\int_{\mathbb{R}^{3}}dw_{\ast
}\int_{S^{2}} B\left( n \cdot \omega ,\left\vert w-w_{\ast }\right\vert
\right) \left[ g^{\prime }g_{\ast }^{\prime }-g_{\ast }g\right] d \omega 
\label{eq:nlB}
\end{gather}
The form of solutions \eqref{B1_0}, namely 
\begin{equation}\label{eq:Ahom}
    g\left(  t,x,w\right) = f\left(  t,v\right) \;  \text{with  $v=w-\xi\left(
	t,x\right)$,}
\end{equation}
is very convenient since the collision operator acts solely on the particles velocities while on the other hand the collisions are invariant under galilean transformations, thus the relative velocity $V$ is not changed by \eqref{eq:Ahom} and one has $Q(f,f)\left(
x,v,t\right)  =Q(g,g)\left(  w,t\right)$.  
Therefore one has only to consider the effect of \eqref{eq:Ahom} on the transport operator in \eqref{eq:nlB}. Plugging \eqref{eq:Ahom} into \eqref{eq:nlB} one obtains
\begin{align*}
\partial_{t}g+w \cdot \partial_{x}g  &  =\partial_{t}f-\sum_{k=1}^{3}\frac
{\partial\xi_{k}}{\partial t}\frac{\partial f}{\partial w_{k}}-\sum_{j=1}%
^{3}\sum_{k=1}^{3}w_{j}\frac{\partial\xi_{k}}{\partial x_{j}}\frac{\partial
f}{\partial w_{k}}\\
&  =\partial_{t}f-\sum_{k=1}^{3}\frac{\partial\xi_{k}}{\partial t}%
\frac{\partial f}{\partial w_{k}}-\sum_{k=1}^{3}\sum_{j=1}^{3}w_{j}%
\frac{\partial\xi_{k}}{\partial x_{j}}\frac{\partial f}{\partial w_{k}}\\
&  =\partial_{t}f-\sum_{k=1}^{3}\frac{\partial\xi_{k}}{\partial t}%
\frac{\partial f}{\partial w_{k}}-\sum_{k=1}^{3}\sum_{j=1}^{3}\left(
v_{j}+\xi_{j}\right) \frac{\partial\xi_{k}}{\partial x_{j}}\frac{\partial
f}{\partial w_{k}}
\end{align*}
Then
\begin{align*}
\partial_{t}g+   w\cdot \partial_{x}g
=\partial_{t}f-\sum_{k=1}^{3}\left[  \left( \frac{\partial\xi_{k}
}{\partial t}+\sum_{j=1}^{3}\xi_{j}\frac{\partial\xi_{k}}{\partial x_{j}%
}\right)  +\sum_{j=1}^{3}v_{j}\frac{\partial\xi_{k}}{\partial x_{j}}\right]
\frac{\partial f}{\partial w_{k}}\ .
\end{align*}
In order to obtain a closed equation depending only on $v,t$ the function
$\frac{\partial\xi_{k}}{\partial x_{j}}$ must be constant. Then $\xi\left(
x,t\right)  =L\left(  t\right)  x$ for some matrix $L\left(  t \right).$ It
is interesting to remark that in terms of $\xi$ the previous condition is just
\begin{align*}
\frac{\partial\xi_{k}}{\partial t}+\sum_{j=1}^{3}\xi_{j}\frac{\partial\xi_{k}%
}{\partial x_{j}}  &  =0\\
\partial_{t}\xi+\xi\cdot\nabla\xi &  =0
\end{align*}
which just means that the "characteristic" momentum $\xi$ is transported with
the speed $\xi$ at each point. Notice that $\xi$ is the same at all the points
and gives the change of average velocity encoded in the solution.

Plugging $\xi(t,x)=L(t)x$ into the second equations determines the form of $L\left(  t\right)$, namely 
\[
\frac{dL}{dt}\left(  t\right)  +\left(  L\left(  t\right)  \right)  ^{2}=0.
\] 
Choosing the initial data $L\left(  0\right)  =A$ it is readily
seen that the explicit solution is 
\begin{equation}
L\left(  t\right)  =\left(  I+tA\right)  ^{-1}A=A\left(  I+tA\right)
^{-1}\ \label{A3}%
\end{equation}
as long as $(I+tA)$ is invertible. 
As mentioned in the introduction, the fact that $L\left(  t\right)$ is always given by this formula follows from the uniqueness theorem for ODEs. Therefore, there are
solutions of \eqref{eq:nlB} with the form \eqref{eq:Ahom} if and only if $L\left(  t\right)  $ has
the form (\ref{A3}). We then obtain the equation:
\begin{equation}
\partial_{t}f-L\left(  t\right)  v\cdot\partial_{v}f=Q(f,f)\left( v\right)  \label{A4}%
\end{equation}

In this paper we studied the linear Rayleigh-Boltzmann equation, in which the collision operator $\Ll$ is related to $Q$ by means of the relation $\Ll(g)=Q(M,g)$, with $M$ the Maxwellian distribution. The above computations suggest that, in order to consider homoenergetic solutions in the linear setting for which $\Ll(g)\left(x,w,t\right)  = \Ll(f)\left(  v,t\right)$,  
also the background gas should be transformed according to \eqref{B1_0}. 

More precisely, we need to consider a transformed Maxwellian distribution
$$ \tilde M_*=\frac 1 {(2\pi)^{ \frac{3}{2}}}\exp\left(-\frac{|w_*-\xi(t,x)|^2}{2}\right).  $$

Set $v_*=w_*-\xi(t,x).$  
Using that $v=w-\xi\left(t,x\right)$ then  we obtain 
$$w-w_* =v+\xi(t,x)-w_*=v-v_*$$
and the velocity collision rule can be rewritten as 
\begin{align*}
    v'& =w'-\xi(t,x) =(v+\xi(t,x)-((v+\xi(t,x)-w_*)\cdot \omega)\omega)-\xi(t,x)\\ &= v-((v-v_*)\cdot \omega)\omega; 
\end{align*}
with $dv_*= dw_*$. Therefore, we arrive at the following equation
\begin{equation}
    \partial_t g -L(t)v \cdot \partial_{v}g= \int_{\R} \int_{S^2} B(n \cdot \omega,|v-v_*|) (M_*'g'-M_*g) d \omega d v_*
\end{equation}
with $M_*=(2\pi)^{-\frac{3}{2}}e^{-\frac{|v_*|^2}{2}}$ and where we have used the form of $\xi(t,x)=L(t)x$ obtained before. 

Therefore the homoenergetic ansatz is consistent if the background is affected by the deformation as well.  
\bigskip

\section{Moment equations (cf. Proposition \ref{lem:derM}) in the case of constant collision kernel} \label{appendixB}

In this appendix we present the moments equations obtained in the case of pseudo Maxwellian interactions in Section \ref{sec:3} in a simpler case. More precisely, we rewrite the counterpart of Proposition \ref{lem:derM} in the case in which the collision kernel satisfies \eqref{eq:AssB1} with $b=1$. Indeed in this case the constants $\alpha, \, \beta$ appearing in \eqref{eq:evolutionM} have an explicit expression. 
\begin{prop}
Let the collision kernel $B$ be as in \eqref{eq:AssB1}  with $b=1$. Let $f \in C([0,+\infty),\M)$ be such that
    $$\int_{\R}|v|^2f(t,dv)< + \infty \quad \text{for all $t \geq 0$} . $$
    Then the moments $M_{jk}$ satisfy the following evolution equation
\begin{gather}\label{eq:evolutionM_bcons}
            \frac{d}{dt}M_{jk}+K(\delta_{j1}M_{2k}+\delta_{k1}M_{2j}) = \left(\frac{1}{2}(3 \alpha-\beta)-\frac{2}{3}\kappa \right) M_{jk}  \notag \\
           + \frac{1}{2}(\beta- \alpha) \sum_{l=1}^3 M_{ll}\delta_{jk}+ \frac{1}{2}(\beta- \alpha)\delta_{jk} + \frac{1}{6}(3\alpha-\beta)\delta_{jk}
    \end{gather}
    where $\delta_{jk}$ is the Kronecker delta and and define 
    \begin{equation}
     \alpha=\int_{S^2} \omega_1^4 d\omega  = \frac{4\pi}{5}, 
        \; \; \beta=\int_{S^2}\omega_1^2 d\omega= \frac{4\pi}{3}.
    \end{equation}.
\end{prop}

\begin{proof}
    Choosing $\vp(v)=v_jv_k$ as test function in \eqref{eq:weakFormTimeDep} and setting $V=v-v_*$ gives
    \begin{gather}\label{eq:moments1_bcons}
        \frac{d}{dt} M_{jk}+K\int_{\R}f(t,dv) v_2\partial_{v_1}(v_jv_k)=\notag \\
        \int_{\R}\int_{\R}\int_{S^2} f(t,dv)M_*\left[(v-(V \cdot \omega)\omega)_j(v-(V \cdot \omega)\omega)_k-v_jv_k\right]dv_* d\omega .
    \end{gather}
Computing the derivative on the left-hand side yields
\begin{gather}
    \frac{d}{dt}M_{jk}+K \int_{\R} f(t,dv)v_2v_j\delta_{j1}+K \int_{\R} f(t,dv)v_2v_k\delta_{k1}= \notag \\
     \frac{d}{dt}M_{jk}+K(\delta_{j1}M_{2k}+\delta_{k1}M_{2j}).
\end{gather}
Deriving an explicit expression for the right-hand is more involved. To to begin with we have that
\begin{gather}
    (v-(V \cdot \omega)\omega)_j(v-(V \cdot \omega)\omega)_k-v_jv_k=-(V \cdot \omega)\omega_jv_k-(V \cdot \omega)\omega_kv_j+(V \cdot \omega)^2\omega_j\omega_k 
\end{gather}
which is a second order tensor in $\omega$. Integrating the linear term and recalling that $V=v-v_*$ yields
\begin{gather}
\int_{\R}\int_{S^2}\int_{\R} f(t,dv)\,  M_*\, \left[-((v-v_*) \cdot \omega)\omega_jv_k-((v-v_*) \cdot \omega)\omega_kv_j \right] dv_* d\omega \notag \\
    = - \int_{\R}f(t,dv)\int_{\R}\int_{S^2}  M_* v_kv_j \omega_j^2 dv_* d\omega-\int_{\R}f(t,dv)\int_{\R}\int_{S^2} M_*v_kv_j \omega_k^2 dv_* d \omega  \notag \\
   = -\frac{8\pi}{3}  \int_{S^2} d \omega \int_{\R}f(t,dv)v_kv_j=-\frac{8\pi}{3} M_{jk}.
\end{gather}
Concerning the second order term, set
\begin{equation}\label{eq:Tjk_bcons}
     \int_{S^2}  (V \cdot \omega)^2\omega_j\omega_k d \omega = T_{jk}=\alpha V_jV_k +\beta|V|^2 \delta_{jk}.
\end{equation}
Since \eqref{eq:Tjk_bcons} is invariant under rotations without loss of generality it is possible to assume $V=|V|(1,0,0)$ which gives
\begin{equation}
    T_{jk} = |V|^2 \int_{S^2}  \omega_1^2\omega_j \omega_k d \omega.
\end{equation}
This yields the system of equations
\begin{equation}
    \begin{cases}
    T_{11} = |V|^2\int_{S^2} \omega_1^4 d \omega = \alpha |V|^2;  & \\
    T_{12} = T_{13}=T_{23}=0; & \\
    T_{22} = |V|^2\int_{S^2}  \omega_1^2\omega_2^2 d\omega ; &  \\
    T_{33} =|V|^2\int_{S^2} \omega_1^2\omega_3^2 d\omega ; & \\
    T_{22}=T_{33}.
\end{cases}
\end{equation}

Summing $T_{22}$ and $T_{33}$ yields
\begin{equation}
    2 T_{22}=T_{22}+T_{33}= |V|^2 \int_{S^2} \omega_1^2(\omega_2+\omega_3^2)d \omega= |V|^2 \int_{S^2}
    \omega_1^2(1-\omega_1^2)d \omega=(\beta-\alpha)|V|^2.
\end{equation}
Hence we have
\begin{gather}
    T=|V|^2\left(
    \begin{array}{ccc}
        \alpha & 0 & 0  \\
         0 & \frac{\beta-\alpha}{2} & 0 \\
         0 & 0 & \frac{\beta-\alpha}{2}
    \end{array} \right)=|V|^2 \frac{\beta-\alpha}{2} I+|V|^2\left(
    \begin{array}{ccc}
        \frac{3 \alpha-\beta}{2} & 0 & 0  \\
         0 & 0 & 0 \\
         0 & 0 & 0
    \end{array} \right)= \notag \\
    |V|^2 \frac{\beta-\alpha}{2} I+ \frac{3 \alpha-\beta}{2} V \otimes V \label{eq:TjkExpl_bcons}. 
\end{gather}
Now integrating   \eqref{eq:TjkExpl_bcons} and recalling that $V=v-v_*$ yields
\begin{align}
    &\int_{\R}f(t,dv)\int_{\R} M_* \left[ \frac{\beta-\alpha}{2}(v-v_*)^2\delta_{jk}+\frac{3 \alpha-\beta}{2}(v-v_*)_j(v-v_*)_k\right]dv_* 
    \notag \\
    & \quad =\left(\frac{\beta-\alpha}{2} \sum_{i=1}^3 M_{ii}+\frac{\beta-\alpha}{2}\right) \delta_{jk}  + \frac{3 \alpha-\beta}{2} M_{jk}+\int_{\R} M_*(v_*)_j(v_*)_k d v_* \notag \\
    &  \quad =
    \left(\frac{\beta-\alpha}{2} \sum_{i=1}^3 M_{ii}+\frac{\beta-\alpha}{2}\right)\delta_{jk} +    \frac{3 \alpha-\beta}{2} M_{jk}+ \frac{3 \alpha-\beta}{6} \delta_{jk}.
\end{align}
Collecting all terms together gives 
\begin{align}
\label{eq:evMjk_bcons}
            &\frac{d}{dt}M_{jk}+K(\delta_{j1}M_{2k}+\delta_{k1}M_{2j})  \notag \\
           =  \left(\frac{1}{2}(3 \alpha-\beta)-\frac{8 \pi}{3} \right) M_{jk} + & \frac{1}{2}(\beta- \alpha) \sum_{i=1}^3 M_{ii}\delta_{jk}+ \frac{1}{2}(\beta- \alpha)\delta_{jk} + \frac{1}{6}(3\alpha-\beta)\delta_{jk}.
\end{align}
Moreover from   \eqref{eq:TjkExpl_bcons} we get the equivalent formulas for $\kappa, \alpha, \beta$, namely 
 \begin{align} \label{eq:alphaBeta_bcons}
    &\alpha=\int_{S^2} \omega_1^4 d\omega =\int_{0}^\pi\cos(\theta)^4\sin(\theta) d\theta =\frac{4\pi}{5}; \nonumber \\&
    \beta=\int_{S^2} \omega_1^2 d\omega=\int_{0}^\pi\cos(\theta)^2\sin(\theta) d\theta =\frac{4\pi}{3},
    \end{align}
   where we used the following parametrization of the sphere $S^2$: 
   \begin{equation*}
\omega=\left(
\begin{array}
[c]{c}%
\cos\left(  \theta\right) \\
\sin\left(  \theta\right)  \cos\left(  \phi\right) \\
\sin\left(  \theta\right)  \sin\left(  \phi\right)
\end{array}
\right)  ,\ \ \theta\in\left[  0,\pi\right]  ,\ \phi\in\left[  0,2\pi\right)
. 
\end{equation*}
\end{proof}

\textbf{Acknowledgements.}
 A.~Nota and N.~Miele are grateful to Bertrand Lods 
for useful discussions on the topic. A.~Nota  gratefully acknowledges the support
by the project PRIN 2022 (Research Projects of National Relevance)- Project code 202277WX43.   J.~J.~L.~Vel\'azquez gratefully acknowledges the support by the Deutsche Forschungsgemeinschaft (DFG)
through the collaborative research centre The mathematics of emerging effects (CRC 1060, Project-ID
211504053) and the DFG under Germany’s Excellence Strategy -EXC2047/1-390685813.
The funders had no role in study design, analysis, decision to publish, or preparation of the manuscript. 
\bigskip

\noindent 
\textbf{Compliance with ethical standards}

\noindent
\textbf{Conflict of interest} The authors declare that they have no conflict of interest.

\bigskip

\def\adresse{
\begin{description}

\item[N. Miele]{ Gran Sasso Science Institute,\\ Viale Francesco Crispi 7, 67100 L’Aquila, Italy  \\
E-mail: \texttt{nicola.miele@gssi.it}}

\item[A. Nota:] {Gran Sasso Science Institute,\\ Viale Francesco Crispi 7, 67100 L’Aquila, Italy  \\
E-mail: \texttt{alessia.nota@gssi.it}}

\item[J.~J.~L. Vel\'azquez] { Institute for Applied Mathematics, University of Bonn, \\ Endenicher Allee 60, D-53115 Bonn, Germany\\
E-mail: \texttt{velazquez@iam.uni-bonn.de}}

\end{description}
}

\adresse

\end{document}